\documentclass[reqno,10pt]{amsart}
\usepackage[dvipsnames,usenames]{color} 
\usepackage[toc,page]{appendix}

\usepackage{tikz}  
\usepackage{mathcomp,wasysym}   

\usepackage{hyperref}  
\usepackage[mathscr]{euscript}  
        
\usepackage{cite}      
\usepackage{graphicx}  
\usepackage[all]{xy} \xyoption{arc} \xyoption{color} 

\usepackage{amsmath} 
\usepackage{amsthm} 
\usepackage{amsfonts}  
\usepackage{amssymb} 

\usepackage{verbatim}

\numberwithin{equation}{section}

\newtheorem{theorem}{\sc Theorem}[section]
\newtheorem{lemma}{\sc Lemma}[section]
\newtheorem{proposition}{\sc Proposition}[section]

\newtheorem{definition}{\sc Definition}[section]

\newcommand{\set}[1]{ \{#1\} } 
   
\newcommand{\bset}[1]{ [#1] }

\newcommand{\pset}[1]{ (#1) }

\newcommand{\Pset}[1]{ \left(#1\right) } 
\newcommand{\norm}[1]{ \|#1\| }

\newcommand{\abs}[1]{ |#1| } 
\newcommand{\Abs}[1]{\left |#1\right| }

\newcommand{\cp}[1]{,_{#1}}

\def\hd{\bar{\partial}}

\def\local{l} 

\def\UL{U_\local}

\def\thetal{\theta_\local}

\def\p{\partial}

\def\eps{{\epsilon}}

\def\uu{u_0^\epsilon}

\numberwithin{figure}{section}

\title[Splash singularity for the Navier-Stokes equatins]
{On the Splash Singularity for the  free-surface of a Navier-Stokes fluid}

\author[D. Coutand]{Daniel Coutand}
\address{CANPDE, Maxwell Institute for Mathematical Sciences and department of Mathematics,
Heriot-Watt University, Edinburgh, EH14 4AS, UK}
\email{D.Coutand@ma.hw.ac.uk}

\author[S. Shkoller]{Steve Shkoller}
\address{Department of Mathematics,
University of California,
Davis, CA 95616}
\email{shkoller@math.ucdavis.edu}


\begin{document}

\begin{abstract} In fluid dynamics,
an interface {\it splash} singularity occurs when a locally smooth interface self-intersects in finite time.  We prove that for $d$-dimensional flows, $d=2$ 
or $3$, the free-surface of a viscous water wave, modeled by the incompressible Navier-Stokes equations with moving free-boundary, has a 
finite-time splash singularity.   In particular, we prove that given a sufficiently smooth initial boundary and divergence-free velocity field, the interface 
will  self-intersect in  finite time.
\end{abstract}

\maketitle

%\tableofcontents

\section{Introduction}
%\subsection{The interface splash singularity}
%The fluid interface  {\it splash singularity}   was introduced by  Castro, C\'{o}rdoba,  Fefferman, Gancedo, \& G\'{o}mez-Serrano in \cite{CaCoFeGaGo2013}.   
%A {\it splash singularity} occurs when a fluid interface remains locally smooth but self-intersects in finite time.     For the two-dimensional water waves problem,  Castro, C\'{o}rdoba,  Fefferman, Gancedo, \& G\'{o}mez-Serrano
%\cite{CaCoFeGaGo2013} showed that a splash singularity occurs in finite time using methods from complex analysis together with a clever transformation of the equations.
%In  Coutand \& Shkoller \cite{CoSh2014}, we   showed the existence of a finite-time splash singularity for the water waves equations in two or three-dimensions (and, more generally, for the one-phase Euler equations), using a very different 
%approach, founded  upon an approximation of the self-intersecting fluid domain by a sequence of smooth fluid domains, each with non  self-intersecting
%boundary. 

\subsection{The interface splash singularity}
The fluid interface  {\it splash singularity}   was introduced by  Castro, C\'{o}rdoba,  Fefferman, Gancedo, \& G\'{o}mez-Serrano in \cite{CaCoFeGaGo2013}
in the context of  the one-phase water waves problem.  As shown in Figure \ref{fig1},
A {\it splash singularity} occurs when a fluid interface remains locally smooth but self-intersects in finite time.  Using methods from complex analysis
together with a clever transformation of the equations,  Castro, C\'{o}rdoba,  Fefferman, Gancedo, \& G\'{o}mez-Serrano
\cite{CaCoFeGaGo2013} showed that a splash singularity occurs in finite time for the
water waves equations.
In  Coutand \& Shkoller \cite{CoSh2014a}, we  showed the existence of a finite-time splash singularity for the one-phase incompressible Euler  equations with free-boundary using a very different 
approach, founded  upon an approximation of the self-intersecting fluid domain by a sequence of smooth fluid domains, each with non  self-intersecting
boundary.     For one-phase flow, it is the vacuum state on one side of the interface which permits this finite-time interface self-intersection, and
neither surface tension nor magnetic fields nor other inviscid regularizations of the interface change this fact  \cite{CaCoFeGaGo2012,CoSh2014a}, and 
even stationary solutions, having a splash singularity, have been shown to exist (see C\'{o}rdoba, Enciso, \& Grubic \cite{CoEnGr}).

\begin{figure}[htbp]
\begin{center}
\includegraphics[scale = 0.4]{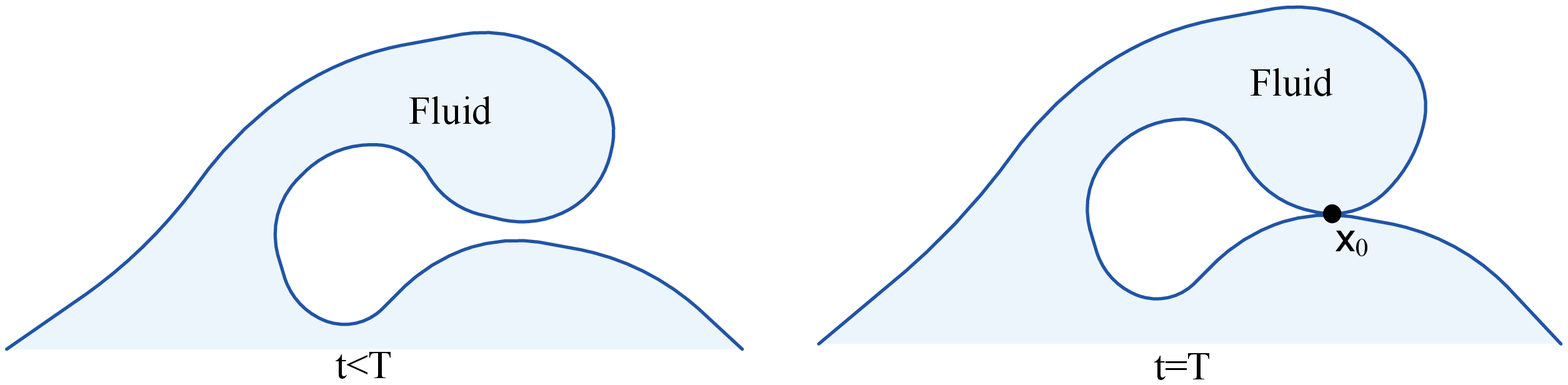}
\caption{The splash singularity at a point $x_0$ occurs when a locally smooth interface self-intersects in finite time
$t=T$.}
\end{center}
\label{fig1}
\end{figure}

On the other hand, for the two-phase incompressible Euler equations, wherein the moving interface is a vortex sheet\footnote{For the vortex sheet problem, it is necessary to have  surface tension in order to ensure well-posedness in Sobolev spaces.},  it was proven by
 Fefferman, Ionescu,  \& Lie \cite{FeIoLi2013} and Coutand \& Shkoller \cite{CoSh2014b} that a splash singularity cannot occur in finite-time while the 
 interface remains locally smooth.   In particular, there is a fundamental difference in the behavior of the fluid interface when vacuum is replaced with fluid
 in the mathematical model.

Since these results have been established for inviscid flows, it is natural to ask if splash singularities can occur for viscous flows modeled by the 
incompressible Navier-Stokes equations with a moving free-surface.    Because the methods of constructing splash singularities for inviscid flows have
relied on the ability to flow backward-in-time, a new strategy must be devised to study the parabolic Navier-Stokes equations.   
By using the change-of-variables employed in \cite{CaCoFeGaGo2013} together with stability estimates,  
Castro, C\'{o}rdoba,  Fefferman, Gancedo, \& G\'{o}mez-Serrano in  \cite{CaCoFeGaGo2015}  have 
shown the existence of finite-time splash singularities for the Navier-Stokes equations.   Herein, we give a different proof which is amenable to any
space dimension $d\ge 2$.

\subsection{The Eulerian description of the Navier-Stokes free-boundary problem}
For $0 \le t \le T$,
the evolution of a $d$-di\-men\-si\-o\-nal ($d=2$ or $3$)  one-phase, incompressible, viscous fluid
with a moving free boundary  is modeled by the in\-com\-pres\-sib\-le Navier-Stokes equations:
\begin{subequations}
  \label{NSe}
\begin{alignat}{2}
u_t+ u\cdot \nabla u +  \nabla  p&= \nu \Delta u  \ \  \ &&\text{in} \ \ \Omega(t) \,,\\
  {\operatorname{div}} u &=0
&&\text{in} \ \ \Omega(t) \,, \\
\nu \operatorname{Def} u \cdot n - p\, n &= 0 \ \ &&\text{on} \ \ \Gamma(t) \,, \\
\mathcal{V} (\Gamma(t))& = u \cdot n &&\ \ \\
u   &= u_0  \ \  &&\text{on} \ \ \Omega(0) \,,\\
   \Omega(0) &= \Omega_0\,.  && 
\end{alignat}
\end{subequations}
The open subset
 $\Omega(t) \subset \mathbb{R}^d  $, $d=2$ or $3$, denotes the time-dependent volume occupied by the fluid,  $\Gamma(t):= \partial\Omega(t)$ denotes
 the moving free-surface, $ \mathcal{V} (\Gamma(t))$ denotes normal
 velocity of $\Gamma(t)$, and $n(t)$ denotes the exterior unit normal vector to the free-surface  $\Gamma(t)$.
  The vector-field $u = (u_1,.., u_d)$ denotes the Eulerian velocity
field, and $p$ denotes the pressure function. We use the notation $\nabla =(\p_1, ..., \p_d)$ to denote the gradient operator, and set
$\operatorname{Def} u = \nabla u +  \nabla u^T$, twice the symmetric part of the gradient of velocity. We have
normalized the equations to have all physical constants equal to 1.

The pressure $p$  is a solution to the following Dirichlet problem:
\begin{subequations}
  \label{p}
\begin{alignat}{2}
- \Delta p  &=  u^i,_j u^j,_i   \ \  \ &&\text{in} \ \ \Omega(t) \,,\\
 p &=  n \cdot \left[ \nu \operatorname{Def} u \cdot n \right] \ \ &&\text{on} \ \ \Gamma(t)  \,,
\end{alignat}
\end{subequations}
so that given an initial domain $\Omega$ and an initial velocity field $u_0$, the initial pressure is obtained as the solution of (\ref{p}) at $t=0$.

\begin{definition}Given a locally smooth, time-dependent fluid interface or free-boundary,   if there exists a time $T< \infty $ such that the
interface $\Gamma(T)$ self-intersects at a point while remaining locally smooth, we call this point of self-intersection at time $T$ a ``splash'' singularity.\end{definition} 

We prove that there exist smooth initial data for the  Navier-Stokes equations (\ref{NSe})  for which such a splash singularity occurs in finite time.

\subsection{Statement of the Main Theorem}\label{sec:maintheorem}

\begin{theorem}[Finite-time splash singularity]\label{theorem_main}  There exist 
\begin{enumerate}
\item  open bounded $C^ \infty $-class initial domains $\Omega\subset \mathbb{R}  ^d$, $d=2$ or $3$,
with $N$ denoting the unit normal vector field on $\p \Omega $, and
\item  smooth divergence-free velocity fields  $u_0$ satisfying the compatibility condition 
$$\left[ \operatorname{Def} u_0 \cdot N\right] \times N =0 \text{ on } \p \Omega \,,$$
\end{enumerate}
such that after a finite time $T^*>0$, the solution to the 
Navier-Stokes equations (\ref{NSe}) has a splash singularity; that is,  the interface $\Gamma(T^*)$ self-intersects.
\end{theorem}
In Theorem \ref{thm_general}, we show that  the geometry of such a splash singularity can be prescribed arbitrarily close (in the $H^3$ norm) to any
sufficiently smooth and prescribed self-intersecting domain.

\subsection{Prior results for the incompressible Navier-Stokes equations with moving free-surface}
Local-in-time  well-posedness of  solutions to (\ref{NSe}) have been known since
the pioneering work of  Solonnikov \cite{Sol1977, Sol1991, Sol1992}; his proof did not rely on energy estimates, but rather on Fourier-Laplace 
transform techniques, which required the use of exponentially weighted anisotropic Sobolev-Slobodeskii spaces with only fractional-order spatial derivatives for the analysis.     Beale \cite{Beale1981} proved local well-posedness in a similar functional framework, and Abels \cite{Abels2005} 
established the existence theory in the $L^p$ Sobolev space framework.   Well-posedness in energy spaces was established by
Coutand \& Shkoller  in \cite{CoSh2002} for the case of surface tension on the free-boundary, and for Navier-Stokes fluid-structure interaction problems
wherein a viscous fluid is coupled to an elastic solid, in
 \cite{CoSh2005,CoSh2006}.  
Guo \& Tice \cite{GuTi2013a} also used energy spaces for local well-posed for the case of zero surface tension.   

Beale \cite{Beale1983} established global existence of solutions to (\ref{NSe})  for small perturbations of equilibrium.   More recent small-data global
existence and decay results (both with and without surface tension) can be found in \cite{TaTa1995}, \cite{PaSo2000}, \cite{NiTeYo2004}, \cite{Hataya2009}, \cite{Bae2011}, and \cite{GuTi2013b,GuTi2013c}.   Recent results on the limit of zero viscosity and the limit of zero surface tension 
can be found in \cite{MaRo2012}, \cite{ElLe2014}, and \cite{WaXi2015}.

For the history of the well-posedness and singularity theory for the inviscid problem, we refer the reader to the introduction in \cite{CoSh2007} and 
\cite{CoSh2014b}.

\textcolor{black}{ 
\subsection{Outline of the paper}
In Section \ref{sec:notation}, we define our notation.   In Section \ref{sec::dino_wave}, we define a sequence of domains $ \Omega^\epsilon $ that we 
use as the initial data
for the splash singularity, wherein the boundary $\Gamma^ \epsilon $ of these domains  is close to self-intersection with a distance $ \epsilon $  between 
two approaching portions of $\Gamma^ \epsilon $.  We convert the Navier-Stokes equations to Lagrangian coordinates in Section \ref{sec::lagrangian}, 
thus fixing the domain.   In Section \ref{sec5}, we present some preliminary lemmas which show that the constant appearing in elliptic estimates and
the Sobolev embedding theorem is independent of $ \epsilon $.   In Section \ref{sec6}, we define the sequence of initial divergence-free velocity fields 
that are guaranteed to satisfy the single compatibility condition that we require, and whose norm is independent of $ \epsilon $.   Section \ref{section7} 
is devoted to the basic a priori estimates for the Navier-Stokes equations in Lagrangian coordinates; following our approach in \cite{CoSh2002}, we
establish estimates for velocity $v \in L^2(0,T;  H^3(\Omega^ \epsilon ) ) \cap C^0([0,T];  H^2(\Omega^ \epsilon ) ) $ which are independent of $ \epsilon $.
We then prove that the vertical component of velocity $v( \cdot t)$ at time $t$ remains in an $O( t^ {\frac{1}{4}} )$ neighborhood
of the vertical component of the initial velocity field.   Using this fact, we prove the main theorem in Section \ref{section8}; we show that by choosing
$ \epsilon $ appropriately, a finite-time splash  singularity must occur at some time $T^* \in (0, 10 \epsilon )$.   We consider a completely arbitrary
geometry for a splash singularity in Section \ref{section9}, by following our definition of a generalized splash domain from our previous work in 
\cite{CoSh2014a}.  This, then, allows us to show in Section \ref{section10}, that we can construct a splash singularity for a geometry which is
arbitrarily close in $H^{3}$ to {\it any} prescribed $H^{3}$ splash domain.
}

\section{Notation, local coordinates, and some preliminary results} \label{sec:notation} 

\subsection{Notation for the gradient vector} \label{sec:grad-horiz-deriv}

Throughout the paper the symbol $\nabla $ will be used to denote the $d$-dimensional gradient vector 
$
\nabla =\left( \frac{\p}{\p x_1}\,,  \frac{\p}{\p x_2}\,, ...,\,  \frac{\p}{\p x_d}  \right)
$.
\subsection{Notation for partial differentiation and  the Einstein summation convention} \label{sec:notat-part-diff}

The $k$th partial derivative of $F$ will be denoted by $F\cp{k}=\frac{
\partial F}{
\partial x_k}$. Repeated Latin indices $i,j,k$, etc., are summed from $1$ to $d$, and repeated Greek indices $\alpha, \beta, \gamma$, etc., are summed from $1$ to $d$$-$$1$. For example, $F\cp{ii}=\sum_{i=1}^d\frac{\p^2F}{\p x_i\p x_i}$, and $F^i\cp{\alpha} I^{\alpha\beta} G^i\cp{\beta}=\sum_{i=1}^d\sum_{\alpha=1}^{d-1}\sum_{\beta=1}^2\frac{\p F^i}{\p x_\alpha} I^{\alpha\beta} \frac{\p G^i}{\p x_\beta}$.

\def\R{ \mathbb{R}  }

\subsection{Tangential (or horizontal) derivatives}\label{sec: tangential derivative} On each boundary
chart  $\UL\cap\Omega$, for $1\le\local\le K$, we let $\bar \p$ denote the \textit{tangential derivative} whose
 $\alpha$th-component given by
\begin{align*}
	\bar \p_ \alpha  f=\Pset{\frac{\p}{\p x_\alpha}\bset{f\circ\thetal}}\circ\thetal^{-1}=\Pset{\pset{ \nabla  f\circ\thetal}\frac{\p\thetal}{\p x_\alpha}}\circ\thetal^{-1} \,.
\end{align*}
For functions defined directly on  $B^+$, $\hd$ is simply the horizontal derivative $\hd = (\partial_{x_1},...,  \partial_{x_{d-1}})$.

\subsection{Sobolev spaces} \label{sec:diff-norms-open}

For integers $k\ge0$ and a bounded domain $U$ of $\R^3$, we define the Sobolev space $H^k(U)$ $\pset{H^k(U;\R^3)}$ to be the completion of $C^\infty(\bar{U})$ $\pset{C^\infty(\bar{U}; \mathbb{R}  ^3)}$ in the norm 
\begin{align*}
	\norm{u}_{k,U}^2=\sum_{\abs{a}\le k}\int_U \Abs{ \nabla ^a u(x) }^2 , 
\end{align*}
for a multi-index $a\in \mathbb{Z}  ^3_+$, with the convention that $\abs{a}=a_1+a_2+a_3$.  When there is no possibility for confusion,
we write $\| \cdot \|_k$ for $\norm{\cdot }_{k,U}$.
For real numbers $s\ge0$, the Sobolev spaces $H^s(U)$ and the norms $\label{n:interior norm}\norm{\cdot}_{s,U}$ are defined by interpolation. 
We will write $H^s(U)$ instead of $H^s(U;\R^d)$ for vector-valued functions.

\subsection{Sobolev spaces on a surface $\Gamma$} \label{sec:sobolev-spaces-gamma} For functions $u\in H^k(\Gamma)$, $k\ge0$, we set 
\begin{align*}
	\norm{u}_{k,\Gamma}^2=\sum_{\abs{a}\le k } \int_\Gamma \Abs{ \hd^a u(x)}^2, 
\end{align*}
for a multi-index $a\in \mathbb{Z}  ^2_+$. For real $s\ge0$, the Hilbert space $H^s(\Gamma)$ and the boundary norm $\label{n:boundary-norm}\abs{\cdot}_s$ is defined by interpolation. The negative-order Sobolev spaces $H^{-s}(\Gamma)$ are defined via duality. That is, for real $s\ge0$, 
$H^{-s}(\Gamma)=H^s(\Gamma)' $.

\subsection{The unit normal and tangent vectors}
We let $n( \cdot ,t)$ denote the outward unit normal vector to the moving boundary $ \Gamma (t)$.   When $t=0$, we let $N_ \epsilon $ denote the
outward unit normal to $\Gamma^ \epsilon $.
For each $ \alpha =1,...,d-1$ and $x \in \Gamma^ \epsilon$ ,  $\tau_ \alpha(x)$ denotes  an orthonormal basis of the ($d$$-$$1$)-dimensional tangent space to
$ \Gamma^\epsilon$ at the point $x$.

\section{The sequence of  initial  domains $\Omega^ \epsilon $}\label{sec::dino_wave}

We shall use, as initial data,  a sequence of domains, whose two-dimensional cross-section resembles a dinosaur neck arching over its body.

\subsection{The ``dinosaur wave'' domains}

\begin{definition}[The domain  $\Omega$]\label{def-dino}
Let $\Omega\subset \mathbb{R}  ^d$, $d=2,3$, be a smooth bounded domain (as shown on  the left of Figure \ref{fig_dino}) with boundary $\Gamma$.
We assume that there are three particular open subsets of $\Omega$ as follows:
\begin{enumerate} 
\item  There exists an open subset $\omega\subset \Omega$ such that its boundary
$\partial\omega\subset\Gamma$ is a vertical   circular cylinder of radius $r$ and  of length $h>0$.

\item There exists an open subset $\omega_+\subset\Omega$ which is the lower-half of an open ball of radius $1$, 
located directly below the cylindrical region $\omega$, and in
contact with the cylindrical region $\overline{\omega}$.     The ``south pole'' of $\omega_+$ is the point $X_+$ (see Figure \ref{fig_initialconditions}).

\item There exists an open subset $\omega_-\subset\Omega$ directly below, at a distance $1$,  from the
 ``south pole'' $X_+$ of $\omega_+$, 
such that the points with  maximal vertical coordinate in $\partial\omega_-\cap \Gamma$ form a subset of the horizontal plane $x_d=0$.

\item Coordinates are assigned to subsets of $\Omega$ as follows:
\begin{enumerate} 
\item  The origin of $ \mathbb{R}  ^d$ is contained in $\partial\omega_-\subset\Gamma \cap \{ x_d=0\}$.
\item The point $X_+$, the  ``south pole'' of $\omega_+$, has the coordinates
$X^ \alpha _+  =0$ for $ \alpha =1,..., d-1$ and $X_+^d =1$.
\item The top boundary of the hemisphere $\omega_+$ is the set $\{ (x_h, x_d) \in \mathbb{R} ^d \ : \ x_d = {\frac{3}{2}} , \  | x_h | < 1 \}$.
\item The cylindrical region $ \omega $ is given by $\{ (x_h, x_d) \in \mathbb{R} ^d \ : \  {\frac{3}{2}} < x_d <{\frac{3}{2}} + h , \  | x_h | < 1 \}$.
\end{enumerate}

\end{enumerate} 

\end{definition}

\begin{figure}[h]
 \begin{tikzpicture}[scale=.5]      
         \draw[color=red,ultra thick] (6,0.5) arc (-90:0:1cm);
         \draw[color=red,ultra thick] (6,0.5) arc (270:180:1cm);
           \draw[color=green,ultra thick] plot[smooth,tension=.6] coordinates{   (5,1.5) (7,1.5)  }; 
            \draw[color=green,ultra thick] plot[smooth,tension=.6] coordinates{   (5,3) (7,3)  }; 
             \draw (6,2) node { $ \omega  $}; 
              \draw (6,1) node { $ \omega_+  $}; 
               \draw (6,-2) node { $ \omega_-  $}; 
         
         \draw[color=blue,ultra thick] (5,3) arc (00:180:2cm);
         \draw[color=blue,ultra thick] (7,3) arc (00:180:4cm);

        \draw[color=blue,ultra thick] plot[smooth,tension=.6] coordinates{   (1,0) (1,2) (1,3) }; 
        \draw[color=blue,ultra thick] plot[smooth,tension=.6] coordinates{   (-1,-3) (-1,2) (-1,3) };

        \draw[color=blue,ultra thick] plot[smooth,tension=.6] coordinates{   (5,1.5) (5,2) (5,3) };
        \draw[color=blue,ultra thick] plot[smooth,tension=.6] coordinates{  (7,3) (7,2) (7,1.5) };
        
        \draw[color=blue,ultra thick] (1,0) arc (180:270:1cm);
        \draw[color=blue,ultra thick] (-1,-3) arc (180:270:1cm);
         
        \draw[color=blue,ultra thick] plot[smooth,tension=.6] coordinates{   (2,-1) (10,-1)  }; 
        \draw[color=blue,ultra thick] plot[smooth,tension=.6] coordinates{   (0,-4) (10,-4)  }; 
        
         \draw[color=blue,ultra thick] (10,-1) arc (90:0:1cm);
         \draw[color=blue,ultra thick] (10,-4) arc (-90:0:1cm);
         
          \draw[color=blue,ultra thick] plot[smooth,tension=.6] coordinates{   (11,-2) (11,-3)  };
          
          \draw (10,-2.5) node { $\Omega $}; 
          \draw (7,5) node { $\Gamma $};

             \draw[color=red,ultra thick] (20,-0.5) arc (-90:0:1cm);
         \draw[color=red,ultra thick] (20,-0.5) arc (270:180:1cm);
          \draw[color=green,ultra thick] plot[smooth,tension=.6] coordinates{   (19,0.5) (21,0.5)  }; 
            \draw[color=green,ultra thick] plot[smooth,tension=.6] coordinates{   (19,3) (21,3)  }; 
             \draw (20,1.5) node { $ \omega^ \epsilon   $}; 
              \draw (20,0.) node { $ \omega_+^ \epsilon   $}; 
              \draw (20,-2.) node { $ \omega_-   $}; 
         
         \draw[color=blue,ultra thick] (19,3) arc (00:180:2cm);
         \draw[color=blue,ultra thick] (21,3) arc (00:180:4cm);

        \draw[color=blue,ultra thick] plot[smooth,tension=.6] coordinates{   (15,0) (15,2) (15,3) }; 
        \draw[color=blue,ultra thick] plot[smooth,tension=.6] coordinates{   (13,-3) (13,2) (13,3) };

        \draw[color=blue,ultra thick] plot[smooth,tension=.6] coordinates{   (19,0.5) (19,2) (19,3) };
        \draw[color=blue,ultra thick] plot[smooth,tension=.6] coordinates{  (21,3) (21,2) (21,0.5) };
        
        \draw[color=blue,ultra thick] (15,0) arc (180:270:1cm);
        \draw[color=blue,ultra thick] (13,-3) arc (180:270:1cm);
         
        \draw[color=blue,ultra thick] plot[smooth,tension=.6] coordinates{   (16,-1) (24,-1)  }; 
        \draw[color=blue,ultra thick] plot[smooth,tension=.6] coordinates{   (14,-4) (24,-4)  }; 
        
         \draw[color=blue,ultra thick] (24,-1) arc (90:0:1cm);
         \draw[color=blue,ultra thick] (24,-4) arc (-90:0:1cm);
         
          \draw[color=blue,ultra thick] plot[smooth,tension=.6] coordinates{   (25,-2) (25,-3)  };
        
         \draw (24,-2.5) node { $\Omega^ \epsilon $}; 
           \draw (21,5) node { $\Gamma^ \epsilon $}; 
           
           \draw(2.5,0.5) -- (4.5,0.5);
           \draw[->] (3.5,2.5) -- (3.5,0.6);
            \draw[->] (3.5,-3.) -- (3.5,-1.1);
           \draw (3.5,-.25) node { $1$};

           \draw(16.5,-0.5) -- (18.5,-0.5);
           \draw[->] (17.5,1.5) -- (17.5,-0.4);
            \draw[->] (17.5,-3.) -- (17.5,-1.1);
           \draw (17.5,-.75) node { $_{ \epsilon }$};

   \draw[color=green,ultra thick] plot[smooth,tension=.6] coordinates{   (2,-1.05) (2,-3.95)  }; 
   \draw[color=green,ultra thick] plot[smooth,tension=.6] coordinates{   (8,-1.05) (8,-3.95)  }; 
   
    \draw[color=green,ultra thick] plot[smooth,tension=.6] coordinates{   (16,-1.05) (16,-3.95)  }; 
   \draw[color=green,ultra thick] plot[smooth,tension=.6] coordinates{   (22,-1.05) (22,-3.95)  };

  \end{tikzpicture} 
%    \caption{{\footnotesize  The ``dinosaur wave'' domains $\Omega^ \delta $ with boundary $\Gamma^ \delta $ is used as initial data for the Navier-Stokes splash singularity.  In order to ensure that a splash occurs, the
%    ``dinosaur neck'' stretches downward so that the distance $ \delta $ between the two portions of the interface becomes sufficiently small.
%On the left, the distance is given as $ \delta _1$ while on the right, the distance is given as $ \delta _2$, and $ \delta _2 < \delta _1$.}}\label{fig_dino}
      \caption{{\footnotesize  {\bf Left.} The ``dinosaur wave'' domain $\Omega $ with boundary $\Gamma $.  {\bf Right.} The sequence of ``dinosaur waves''
      $\Omega^ \epsilon $ with boundary $\Gamma^ \epsilon $, $ \epsilon >0$, 
       used as initial data for the Navier-Stokes splash singularity.  In order to ensure that a splash occurs, the
    ``dinosaur neck'' $ \omega ^ \epsilon $ stretches downward so that there is a  distance $ \epsilon  $ between the two portions.   The domains $\Omega^\epsilon $ simply stretch the neck
    of the dinosaur, and are identical to $\Omega$ away from the neck.
     }}\label{fig_dino}
\end{figure}
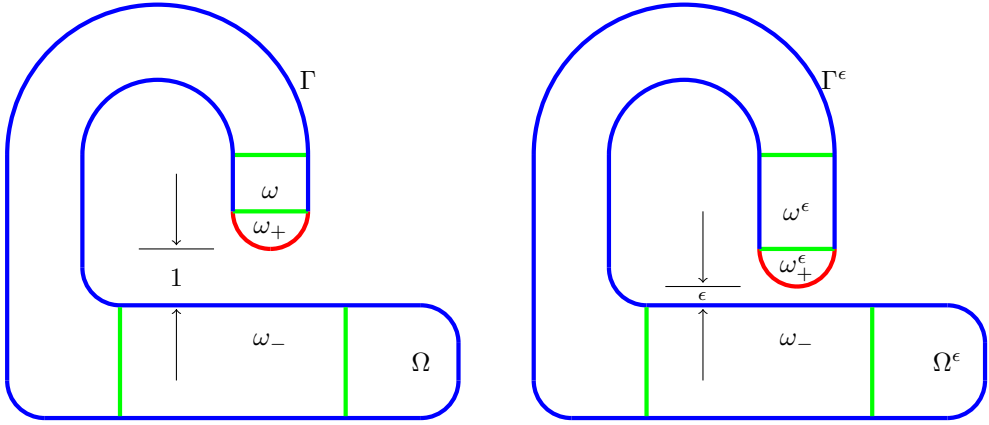

\begin{definition}[The initial domains $\Omega^ \epsilon $]\label{def-dino-e} For $ 0 < \epsilon \ll 1$, 
let $\Omega\subset \mathbb{R}  ^d$, $d=2,3$, be a smooth bounded domain (as shown on  the right of Figure \ref{fig_dino}) with boundary 
$\Gamma^ \epsilon $.
We define the domain $ \Omega^\epsilon $ to be the following modification of the domain $\Omega$:
\begin{enumerate} 
\item  There exists an open subset $\omega^ \epsilon \subset \Omega^ \epsilon $, which is a vertical dilation of the domain $\omega$,
such that its boundary
$\partial\omega^ \epsilon \cap \Gamma^ \epsilon $ is a vertical   circular cylinder of radius $r$ and  of length $h+1- \epsilon $.

\item There exists an open subset $\omega_+^ \epsilon \subset\Omega^ \epsilon $ which is the set $\omega^+$ translated vertically
downward a distance $1- \epsilon $; hence, $\omega_+^ \epsilon $  is the lower-half of an open ball of radius $1$, 
located directly below the cylindrical region $\omega^ \epsilon $, and in
contact with the cylindrical region $\overline{\omega^ \epsilon }$.     The ``south pole'' of $\omega_+^ \epsilon $ is the point $X_+^ \epsilon $.

\item There exists an open subset $\omega_-\subset\Omega^ \epsilon $ directly below, and a distance $ \epsilon $,  from the
 ``south pole'' $X_+^ \epsilon $ of $\omega_+^ \epsilon $, 
such that the points with  maximal vertical coordinate in $\partial\omega_-\cap \Gamma$ form a subset of the horizontal plane $x_d=0$.
We assume that $\partial\omega_-\cap \Gamma$ contains a $d$$-$$1$-dimensional ball of radius $\sqrt{ \epsilon }$.

\item Coordinates are assigned to subsets of $\Omega^ \epsilon $ as follows:
\textcolor{black}{ 
\begin{enumerate} 
\item  The origin of $ \mathbb{R}  ^d$ is contained in $\partial\omega_-\subset\Gamma \cap \{ x_d=0\}$.
\item The point $X_+^ \epsilon $, the  ``south pole'' of $\omega_+^ \epsilon $, has the coordinates
$X^ \alpha _+  =0$ for $ \alpha =1,..., d-1$ and $X_+^d = \epsilon $.
\item The top boundary of the hemisphere $\omega_+^ \epsilon $ is the set $\{ (x_h, x_d) \in \mathbb{R} ^d \ : \ x_d = \epsilon + {\frac{1}{2}} , \  | x_h | < 1 \}$.
\item The cylindrical region $ \omega^ \epsilon  $ is given by $\{ (x_h, x_d) \in \mathbb{R} ^d \ : \  \epsilon + {\frac{1}{2}}  < x_d < \epsilon + {\frac{1}{2}}  + h , \  | x_h | < 1 \}$.
\end{enumerate} 
}
\end{enumerate} 

\end{definition}

\subsection{Local coordinate charts for $\Omega$ and $\Omega^ \epsilon $ } \label{sec::charts}    
\subsubsection{Local charts for $\Omega$}
We let $s \ge 3$ and $0 < \epsilon \ll 1$.
Let $\Omega\subset \mathbb{R}^d  $ denote a smooth  open set, and let $\{U_l\}_{l=1}^K$ denote an open covering of $\Gamma=\p\Omega$, such that for each $l\in \{1,2,\dots,K\}$, with 
\begin{align*}
	B&=B(0,1),\text{ denoting the open ball of radius $1$ centered at the origin and}, \\
	B^+&=B\cap\set{x_d>0}, \\
	B^0&=\overline B\cap\set{x_d=0}, 
\end{align*}
there exist $C^\infty $ charts $\thetal$ which satisfy 
\begin{subequations}
\label{normalchart}
\begin{align}
	\thetal\colon B\to\UL\ &\text{ is an $C^ \infty $ diffeomorphism},  \\
	\thetal(B^+)&=\UL\cap\Omega, \ \ \
	\thetal(B^0)=\UL\cap\Gamma\,,
\end{align}
\end{subequations}
and $\det \nabla \thetal=C_l$ for a constant $C_l >0$. We assume these boundary charts can be split into three categories (each being non empty):

\begin{itemize} 
\item  For $1\le l\le K_1$, $\theta_l (B^+) \subset \omega $.
\item  For $K_1+1\le l\le K_2$, $\theta_l (B^+) \not \subset \omega $ and $\theta_l (B^+)\cap \omega_+ =\emptyset$.
\item For $K_2+1\le l\le K$, $\theta_l (B^+) \not \subset \omega $ and $\theta_l (B^+)\cap \omega_+ \neq \emptyset$.
\end{itemize} 

We also assume that the images of any charts $\theta_l$ for $K_1+1\le l\le K_2$ does not intersect any of the images of the charts for $K_2+1\le l\le K$.
%
%
%\begin{figure}
%	[here] \centering 
%	\includegraphics[scale = 0.8]{fig_cover.eps}
%	\caption{Indexing convention for the open cover $\set{U_\local}_{\local=1}^L$ of $\Omega$.} 
%\end{figure}

Next, for $L>K$, we let $\set{\UL}_{\local=K+1}^L$ denote a family of open sets contained in $\Omega$ such that $\set{\UL}_{\local=1}^L$ is an open cover of $\Omega$ and there exist smooth diffeomorphisms $\theta_l:B \to U_l$ with  $\det \nabla \theta_l$ equal to a constant $C_l>0$.

Just as for the case of the boundary charts, we assume that these interior charts are split into three categories (each being non empty):
\begin{itemize} 
\item  For $K+1\le l\le L_1$, $\theta_l (B) \subset \omega $.
\item For $L_1+1\le l\le L_2$, $\theta_l (B) \not \subset \omega $ and $\theta_l (B^+)\cap \omega_+ =\emptyset$.

\item  For $L_2+1\le l\le L$,  $\theta_l (B) \not \subset \omega $ and $\theta_l (B^+)\cap \omega_+ \neq \emptyset$.
\end{itemize} 
We assume that the union of the images of the charts $\theta_l$, for $1\le l\le K_1$ and $K+1\le l\le L_1$ contains the shortened cylindrical region
$\stackrel{ \circ }{\omega} =    \{ (x_h, x_d) \in \mathbb{R} ^d \ : \  {\frac{3}{2}} +  \epsilon  < x_d <{\frac{3}{2}} +   h - \epsilon  , \  | x_h | < 1 \}$.

%%with same center and radius as $\omega$, but with height $\frac{1-3\epsilon}{1-\epsilon}\ h$.
%\textcolor{red}{ 
%We also assume that the union of the images of the charts $\theta_l$, for $K_2+1\le l\le K$ and $L_2+1\le l\le L$ contains the  
%shortened cylindrical region $\tilde \omega =    \{ (x_h, x_d) \in \mathbb{R} ^d \ : \   {\frac{3}{2}} + 2 \epsilon  < x_d <{\frac{3}{2}} + h -  2\epsilon  , \  | x_h | < 1 \}$.
%}
%%parts of $\omega$ at a distance less than or equal to a distance of $2\epsilon$ from the vertically maximum and minimum points of $\omega$.
%\textcolor{red}{ 
%Finally, we assume the images of any of the charts $\theta_l$ for $L_1+1\le l\le L_2$ does not intersect any of the images of the charts $\theta_l$ for $L_2+1\le l\le L$.
%}

\textcolor{black}{
We assume that the union of the images of the charts $\theta_l$, for $1\le l\le K_1$ and $K+1\le l\le L_1$ contains the shortened cylindrical region
$$\stackrel{ \circ }{\omega} =    \{ (x_h, x_d) \in \mathbb{R} ^d \ : \  {\frac{3}{2}} +\frac{h}{3+h} \frac{h}{2}   < x_d <{\frac{3}{2}} +   h - \frac{h}{3+h}\frac{h}{2}  , \  | x_h | < 1 \}$$
 of length $\frac{3}{3+h} h$}

%with same center and radius as $\omega$, but with height $\frac{1-3\epsilon}{1-\epsilon}\ h$.
\textcolor{black}{ 
We also assume that the union of the images of the charts $\theta_l$, for $K_2+1\le l\le K$ and $L_2+1\le l\le L$ contains the complement in 
$\omega$ of the 
shortened cylindrical region 
$$\tilde \omega =    \{ (x_h, x_d) \in \mathbb{R} ^d \ : \   {\frac{3}{2}} + \frac{h}{2}\frac{\frac{3}{2}+h} {3+h}< x_d <{\frac{3}{2}} + h - \frac{h}{2}\frac{\frac{3}{2}+h} {3+h}  , \  | x_h | < 1 \}$$ 
of length $\frac{3}{3+h}\frac{h}{2}$, so that the complement is of length $\frac{\frac{3}{2}+h}{3+h} h$.
}
%parts of $\omega$ at a distance less than or equal to a distance of $2\epsilon$ from the vertically maximum and minimum points of $\omega$.

\textcolor{black}{ 
Finally, we assume the images of any of the charts $\theta_l$ for $L_1+1\le l\le L_2$ does not intersect any of the images of the charts $\theta_l$ for $L_2+1\le l\le L$.
}

\subsubsection{Local charts for $\Omega^ \epsilon $}
We next explain how this system of charts can be simply modified to describe $\Omega^\epsilon$ using the following three steps:

\begin{enumerate}
\item For either $1\le l\le K_1$ or $K+1\le l\le L_1$, we define the vertically dilated chart 
(corresponding to a cylinder with length dilated from $h$ to $h+1-\epsilon$) 
$$\theta_l^\epsilon=\left(\theta_1,\theta_2, \frac{h+1-\epsilon}{h} (\theta_3-{\frac{3}{2}} )+ {\frac{1}{2}} +\epsilon\right)\,.$$ 

Note that $\theta_l^\epsilon$ sends any point whose image by $\theta_l$ was  at the altitude $ {\frac{3}{2}} $ in $\overline{\omega}$ (respectively 
${\frac{3}{2}} +h$) into a point of altitude ${\frac{1}{2}} +\epsilon$ (respectively ${\frac{3}{2}} +h$) in $\Omega^\epsilon$.
\item  For either $K_1+1\le l\le K_2$ or $L_1+1\le l\le L_2$ , we set 
$\theta_l^\epsilon=\theta_l$.
\item  For either $K_2+1\le l\le K$ or $L_2+1\le l\le L$, we set the translated in the vertical direction chart $\theta_l^\epsilon=\theta_l-(1-\epsilon) e_d$.
\end{enumerate}
These charts describe $\Omega^\epsilon$, and again $\det \nabla \theta_l$ is a strictly positive constant given by either $C_l$ or $\frac{h+1-\epsilon}{h} C_l$.

\subsubsection{Cut-off functions on charts covering $\Omega$}

\textcolor{black}{ 
Let $\{\xi_l\}_{l=1}^L$ denote a smooth partition of unity, subordinate to the covering $\{U_l\}_{l=1}^L$; 
i.e., $ \xi _l \in C^ \infty _c(U_l)$,  $0 \le \xi _l \le 1$, and  $\sum_{l=1}^L \xi_l =1$.}

\textcolor{black}{ 
 We set   $ \mathcal{B}_l = B^+ $ for $l=1,...,K$, and $ \mathcal{B}_l = B$ for $l=K+1,...,L$.   
For each $l=1,...,L$, we set $\zeta_l = \xi _l \circ \theta_l$, so that $ \zeta_l \in C^ \infty _c ( \mathcal{B} _l)$ whenever the charts $\theta_l$ are smooth.}

\subsubsection{Cut-off functions on charts covering $\Omega^ \epsilon $}\label{sec::partition}   
\textcolor{black}{ We define the cut-off functions $\xi _l^ \epsilon $  as follows: 
$$
\xi ^ \epsilon _l \circ \theta_l^ \epsilon = \xi _l \circ \theta_l \,.
$$
Setting $ \zeta _l = \xi ^ \epsilon _l \circ \theta_l^ \epsilon$, we see that $\| \zeta_l \|_{k, \mathcal{B} _l} $ is bounded by a constant which is independent 
of $ \epsilon $.
}

\section{The Lagrangian description of the Navier-Stokes free-boundary problem}  \label{sec::lagrangian}
For $ \epsilon >0$, 
 we let $\Omega^ \epsilon $ with boundary $\Gamma^ \epsilon $ be given by Definition \ref{def-dino-e},
and we transform the system (\ref{NSe}) into a system of equations set on this reference domain.   To do so, we shall employ the
Lagrangian coordinates.

The Lagrangian flow map $\eta ( \cdot ,t)$ is the solution of the
$\eta_t (x,t) = u(\eta(x,t),t)$ for $t>0$ with initial condition $\eta(x,0) =0$.  Since $ \operatorname{div} u=0$, it follows that $\det \nabla \eta =1$.
For each instant of time $t$ for which the flow is well-defined, we have 
$$
\eta( \cdot ,t): \Omega^ \epsilon   \to \Omega (t) \text{ is a diffeomorphism}\,;
$$
furthermore, thanks to (\ref{NSe}d), 
$$
\Gamma(t) = \eta( \Gamma^ \epsilon , t) \,.
$$
Notationally, we keep the dependence on $ \epsilon >0$ implicit, except for the initial domain and boundary.

Next, we define
\begin{align*}
v &= u \circ \eta   \text{ (Lagrangian velocity)},  \\
q&=p \circ \eta   \text{ (Lagrangian pressure)}, \\
A &= [ \nabla  \eta]^{-1}  \text{ (inverse of the deformation tensor)}\,, \\
g_{ \alpha \beta } &=  \eta, _ \alpha \cdot \eta,_\beta \ \ \alpha ,\beta =1,.., d-1 \text{ (induced metric on $\Gamma$)}\,, \\
\mathfrak{g}  & = \det( g_{ \alpha \beta }) \,.
\end{align*}
We also define the Lagrangian analogue of some of the fundamental differential operators present in this equation:
\begin{align*}
\operatorname{div} _\eta v &=  (\operatorname{div} u) \circ \eta = v^i,_j A^j_i  \,, \\
\operatorname{curl} _\eta v &=  (\operatorname{curl} u) \circ \eta \text{ or }  [ \operatorname{curl} _ \eta v]_i = \varepsilon_{ i jk}  v^k,_r A^r_j \,,  \\
\operatorname{Def} _\eta v &=  (\operatorname{Def} u) \circ \eta  \text{ or } [\operatorname{Def} _ \eta v]^i_j =  v^i,_rA^r_j + v^j,_rA^r_i \,, \\
\Delta _\eta v &=  (\Delta u) \circ \eta = ( A^j_r A^k_r v,_k),_j \,.
\end{align*}

The Lagrangian version of equations (\ref{NSe})  is given on
the fixed reference domain $\Omega^ \epsilon $ by
\begin{subequations}
\label{NSlag}
\begin{alignat}{2}
\eta( \cdot ,t) & = e + \int_0^t v(\cdot ,s) ds  \ && \text{ in } \Omega^ \epsilon  \times [0,T] \,,  \\
 v_t  + A^T \nabla q   &= \nu \Delta _ \eta v \ \ && \text{ in } \Omega^ \epsilon  \times (0,T] \,,  \\
\operatorname{div} _\eta v &=0 \ \ && \text{ in } \Omega^\epsilon  \times [0,T] \,,  \\
\nu \operatorname{Def} _ \eta v \cdot n - qn &=0\ \ && \text{ on } \Gamma^ \epsilon  \times [0,T] \,,  \\
(\eta,v)  &=(e,u_0) \ \  \ \ && \text{ in } \Omega^\epsilon  \times \{t=0\} \,, 
\end{alignat}
\end{subequations}
where $e(x)=x$ denotes the identity map on $\Omega$, and where we write $n$ for $n(\eta)$ in the Lagrangian description; in particular, the
unit normal vector $n$ at the point $\eta( x, t)$ can be expressed in terms of the cofactor matrix $A$ and the time $t=0$ normal vector $N_ \epsilon $ as
$$
n = A^T N_ \epsilon / |  A^T N_ \epsilon| \,.
$$
 Due to (\ref{NSlag}c),
$$
\Delta _ \eta v = \operatorname{div} _ \eta \operatorname{Def} _\eta v \,,
$$
so that  (\ref{NSlag}d) can be viewed as the natural boundary condition.    The variables $ \eta, v$, and $q$ have an a priori dependence on $ \epsilon 
> 0$, but we do not explicitly write this.

Local-in-time  existence and uniqueness of solutions to (\ref{NSlag}) have been known since
the pioneering work of  Solonnikov \cite{Sol1977}.   We shall establish a priori estimates for (\ref{NSlag}) with the initial domain $\Omega^ \epsilon $ and
with  divergence-free initial velocity fields 
 satisfying the single compatibility condition 
\begin{equation}\label{comp}
[ \operatorname{Def} u^ \epsilon _0 \cdot N^ \epsilon ]\cdot \tau^ \epsilon _ \alpha  =0 \text{  on  } \Gamma^ \epsilon  \,, 
\end{equation} 
where $N^ \epsilon $ denotes the outward unit
normal to $\Gamma^ \epsilon $ and $\tau^ \epsilon _ \alpha $, $ \alpha =1,.., d-1$, denotes the $d$$-$$1$ tangent vectors to $\Gamma^ \epsilon $.

We will show that both the a priori estimates and
the time of existence for solutions are independent of the distance $ \epsilon  >0$ between the falling dinosaur head $X_+^ \epsilon $ and the flat trough
$ \partial \omega_- \cap \{ x_d=0\}$ (see Figure
\ref{fig_dino}).   To do so, we shall rely on some basic lemmas that provide us constants which are independent of $ \epsilon$.

\section{Elliptic and  Sobolev constants are independent of $\epsilon $}\label{sec5}
We consider the following linear Stokes problem
\begin{subequations}
\label{Stokes}
\begin{alignat}{2}
- \Delta u + \nabla p& = f  \ && \text{ in } \Omega^ \epsilon\,,  \\
\operatorname{div} u &= \phi  \ \ && \text{ in } \Omega^ \epsilon  \,,  \\
u&=g\ \ && \text{ on } \Gamma^ \epsilon \,, 
\end{alignat}
\end{subequations}

\begin{lemma}[Estimates for the Stokes problem on $\Omega^ \epsilon $] \label{lemma1} 
Suppose that  for  integers $k \ge 2$, $f \in H^{k-2}(\Omega^ \epsilon )$, $\phi \in H^{k-1}(\Omega^ \epsilon )$, and $g \in H^{k-1/2}(\Gamma^\epsilon )$,
and $\int_{ \Omega ^ \epsilon } \phi(x) dx = \int_{\Gamma^\epsilon } g \cdot N\, dS$.  
Then, there exists a unique solution $u \in H^k (\Omega^ \epsilon )$ 
and $p \in H^{k-1} (\Omega^ \epsilon )/ \mathbb{R}  $ to the Stokes problem (\ref{Stokes}).  Moreover, there is a constant $C$ depending only on 
$\Omega$, but  independent of $ \epsilon >0$, such that
\begin{equation}\label{stokes-reg}
\|u \|_k + \|p\|_{k-1} \le C \left( \|f\|_{k-2} + \|\phi\|_{k-1} + | g|_{k-1/2} \right) \,.
\end{equation}
\end{lemma} 
\begin{proof} 
The estimate (\ref{stokes-reg}) is well-known on the domain $\Omega$; see, for example, \cite{AmGi1991}.  This estimate on the sequence of domains
$\Omega^ \epsilon $ follows by localization using the charts $\theta^ \epsilon _l$ given in Section \ref{sec::charts}.  Since the charts $\theta^ \epsilon _l$, are modified from the charts $\theta_l$ by a vertical dilation with lower and upper bound that is uniform in $ \epsilon $, the constant for the elliptic 
estimate in each chart is independent of $ \epsilon>0$.
\end{proof} 

\begin{lemma}[Sobolev constant on $\Omega^ \epsilon $] \label{lemma2} Independent of $ \epsilon$, there exists a constant $C>0$ which depends only
on the domain $\Omega$, such that
$$
\max_{x\in\Omega^ \epsilon } | u(x)| \le C \|u\|_{s, \Omega^ \epsilon } \ \ \forall  u \in H^s(\Omega^ \epsilon ) \,, \ \ s> d/2 \,.
$$
\end{lemma}
\begin{proof} 
The constant is determined by the radius $r$ of the smallest ball $B(x,r)$ for $x \in \Omega^ \epsilon $, such that 
$B(x,r) \subset \overline{ \Omega ^ \epsilon }$.  By Definition \ref{def-dino-e} of the domains $\Omega^ \epsilon $, $r$ does not depend on $ \epsilon $,
and hence the Sobolev constant $C$ only depends on $\Omega$.
\end{proof}  

\begin{lemma}[Trace theorem on $\Omega^ \epsilon $]\label{lemma3}
Independent of $ \epsilon$, there exists a constant $C>0$ which depends only
on the domain $\Omega$, such that for $ s\in( {\frac{1}{2}},3] $
$$
\|u\|_{s-{\frac{1}{2}} , \Gamma^ \epsilon } \le C \|u\|_{s , \Omega^ \epsilon }  \ \ \forall u \in H^s( \Omega^ \epsilon ) \,.
$$
\end{lemma} 
\begin{proof} 
From the standard trace theorem in $B^+$,  we have the existence of a constant $C>0$ such that for any boundary chart,
$$
\|u\circ\theta_l^\epsilon\|_{s-{\frac{1}{2}} , B_0 } \le C \|u\circ\theta_l^\epsilon\|_{s , B^+ }  \ \ \forall u \in H^s( \Omega^ \epsilon ) \,.
$$
Now,  since $\theta_l^\epsilon$ is either a chart $\theta_l$ for the domain $\Omega$  or a vertical dilation of such a chart with a uniform bounded from below and above as is made precise in Section \ref{sec::charts}, this implies that by the chain rule, 
$$
\|u\|_{s-{\frac{1}{2}} , \theta_l^\epsilon(B_0) } \le C \|u\|_{s , \theta_l^\epsilon(B^+) }  \ \ \forall u \in H^s( \Omega^ \epsilon ) \,.
$$ 
Since $\Gamma^\epsilon$ is the union of all $\theta_l^\epsilon(B_0)$, $1\le l\le K$,  the above inequality implies the result.
\end{proof} 

%\begin{lemma}[Extension theorem on $\Omega^ \epsilon $]\label{lemma4}  Let $O\subset \mathbb{R}  ^d $ be a bounded open set, and let $K \subset
%O$ be compact such that $ \Omega^ \epsilon \subset K \subset O$.  Then, for $k\ge 0$ there exists a  Sobolev extension operator 
%$E: H^k(\Omega^ \epsilon ) \to
%H^k( \mathbb{R}  ^d)$ such that for each $u\in H^k ( \Omega^\epsilon )$, $Eu =u$ a.e. in $\Omega^ \epsilon $, the support of $Eu$ is contained in $O$
%and
%for a constant $C>0$ which depends only
%on the domain $\Omega$, 
%$$
%\|Eu\|_{k, \mathbb{R}  ^d } \le  C\|u\|_{k , \Omega^ \epsilon }  \ \ \forall u \in H^k( \Omega^ \epsilon ) \,.
%$$
%\end{lemma} 
%\begin{proof} 
%\end{proof} 

\section{The sequence of initial velocity fields $u_0^ \epsilon $}\label{sec6}
\subsection{Constructing the sequence of initial velocity fields $\uu$}\label{sec::u0}
As described in Definition \ref{def-dino-e}, 
near the intended splash (or self-intersection) point, the open set $\Omega^ \epsilon $ consists of two sets: the upper set $\omega^ \epsilon _ +$ and the lower set 
$\omega_-$ whose boundary  contains the flat ``dinosaur belly'' at $ x_d=0$, as shown in Figure \ref{fig_initialconditions}.   We 
We let $X_+^ \epsilon $ denote the point which has the smallest vertical coordinate in $\partial \omega^ \epsilon _+$.    Directly below, we let
$X_-$ be the point in $\partial \omega_- \cap \{ x_d=0\}$ with the same horizontal coordinate as $X_+^ \epsilon $.   Without loss of generality, we
set $X_-$ to be the origin of $ \mathbb{R}  ^d$.
 
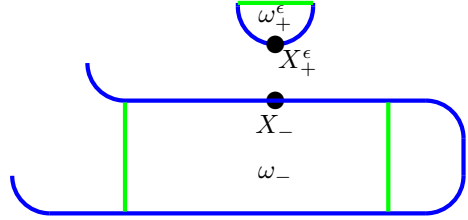
\begin{figure}[h]
 \begin{tikzpicture}[scale=.5]      
         \draw[color=blue,ultra thick] (6,0.5) arc (-90:0:1cm);
         \draw[color=blue,ultra thick] (6,0.5) arc (270:180:1cm);
%           \draw[color=green,ultra thick] plot[smooth,tension=.6] coordinates{   (5,1.5) (7,1.5)  }; 
%            \draw[color=green,ultra thick] plot[smooth,tension=.6] coordinates{   (5,3) (7,3)  }; 
             
%         
%         \draw[color=blue,ultra thick] (5,3) arc (00:60:2cm);
%         \draw[color=blue,ultra thick] (7,3) arc (00:60:4cm);

%        \draw[color=blue,ultra thick] plot[smooth,tension=.6] coordinates{   (1,0) (1,2) (1,3) }; 
%        \draw[color=blue,ultra thick] plot[smooth,tension=.6] coordinates{   (-1,-3) (-1,2) (-1,3) };

        \draw[color=blue,ultra thick] plot[smooth,tension=.6] coordinates{   (5,1.5)  (5,1.6) };
        \draw[color=blue,ultra thick] plot[smooth,tension=.6] coordinates{  (7,1.6)  (7,1.5) };
         \draw[color=green,ultra thick] plot[smooth,tension=.6] coordinates{  (7,1.6) (5,1.6) };
        
        \draw[color=blue,ultra thick] (1,0) arc (180:270:1cm);
        \draw[color=blue,ultra thick] (-1,-3) arc (180:270:1cm);
         \draw (6,1.2) node { $\omega_+^ \epsilon $}; 
          \draw (6.,.5) node { $\newmoon$}; 
          \draw (6.6,0) node { $X_+^ \epsilon $}; 
           
             \draw (6.,-1) node { $\newmoon$}; 
               \draw (6,-1.7) node { $X_- $}; 
         
        \draw[color=blue,ultra thick] plot[smooth,tension=.6] coordinates{   (2,-1) (10,-1)  }; 
        \draw[color=blue,ultra thick] plot[smooth,tension=.6] coordinates{   (0,-4) (10,-4)  }; 
        
         \draw[color=blue,ultra thick] (10,-1) arc (90:0:1cm);
         \draw[color=blue,ultra thick] (10,-4) arc (-90:0:1cm);
         
          \draw[color=blue,ultra thick] plot[smooth,tension=.6] coordinates{   (11,-2) (11,-3)  };
          
          \draw (6,-3) node { $\omega_-  $}; 
          \draw[color=green,ultra thick] plot[smooth,tension=.6] coordinates{   (2,-1.05) (2,-3.95)  }; 
   \draw[color=green,ultra thick] plot[smooth,tension=.6] coordinates{   (9,-1.05) (9,-3.95)  }; 
%          \draw (7,5) node { $\Gamma $}; 

  \end{tikzpicture} 
      \caption{{\footnotesize  In a neighborhood of the intended splash point, we suppose that $\Omega^ \epsilon $ consists of two sets:  the upper set $\omega_+^ \epsilon $
      and  the lower set $\omega_-$ containing the horizontally flat ``dinosaur belly.''   The point $X_+^ \epsilon $ is at a distance $ \epsilon $ from the
      set $\omega_-$ and the point $X_-$ is assumed to be the origin in $ \mathbb{R}  ^d$.
     }}\label{fig_initialconditions}
\end{figure}

We choose a  smooth function $b_0^\epsilon \in C^ \infty ( \Gamma ^ \epsilon ) $ such that $b_0^ \epsilon = -1$  in a small neighborhood of 
$X_+^ \epsilon $ on $\partial \omega_+^ \epsilon $, $b_0^ \epsilon =0$ on $\partial \omega_- $,  $b_0^ \epsilon =0$ on $\partial \omega^ \epsilon\cap \Gamma^\epsilon  $,  $\int_{ \Gamma ^ \epsilon } b_0^ \epsilon \, dS=0$, and satisfying  the estimate
\begin{equation}\label{b-est}
\|b_0^ \epsilon \|_{2.5,\Gamma^ \epsilon } \le m_0 < \infty   \,,
\end{equation} 
where $m_0$ does not depend on $ \epsilon $.

We define the initial velocity field $u_0^ \epsilon $ at $t=0$ as the solution to the following Stokes problem:
\begin{subequations}
\label{Stokes2}
\begin{alignat}{2}
- \Delta u_0^ \epsilon  + \nabla r_0^ \epsilon & = 0  \ && \text{ in } \Omega^ \epsilon\,,  \\
\operatorname{div} u_0^ \epsilon  &= 0  \ \ && \text{ in } \Omega^ \epsilon  \,,  \\
[\operatorname{Def} u_0^ \epsilon \cdot N^ \epsilon ] \cdot \tau _ \alpha ^ \epsilon  &=0 \ \ && \text{ on } \Gamma^ \epsilon \,, \\
u_0^ \epsilon \cdot N^ \epsilon  &=b^ \epsilon \ \ && \text{ on } \Gamma^ \epsilon \,, 
\end{alignat}
\end{subequations}
with $N^ \epsilon $ denoting the outward unit normal to $ \Gamma  ^ \epsilon $ and $\tau_ \alpha ^ \epsilon $,  $\alpha =1,2$ denoting an orthonormal
basis of the tangent space to
 $\Gamma ^ \epsilon $ (if the dimension $d=2$, then there is only one tangent vector).
Using the  regularity theory of this elliptic system (see, for example, \cite{SoSc1973} or  \cite{AmSe2011} and references therein), together with the proof 
of Lemma \ref{lemma1}, for a constant independent of $ \epsilon >0$, 
\begin{equation}\label{u0-est}
\|u_0^ \epsilon \|_{3, \Omega^ \epsilon } \le C \|b^ \epsilon \|_{2.5,\Gamma^ \epsilon } \le C \, m_0\,.
\end{equation} 
The boundary condition (\ref{Stokes2}c) ensures that  $u_0^ \epsilon $ satisfies (\ref{comp}).

\subsection{The initial pressure function $p_0^ \epsilon $}

The initial pressure function $p_0^ \epsilon $ at $t=0$ then satisfies
\begin{subequations}
  \label{p0}
\begin{alignat}{2}
- \Delta p_0^ \epsilon   &=  (u_0^ \epsilon )^i,_j (u_0^ \epsilon )^j,_i   \ \  \ &&\text{in} \ \ \Omega^ \epsilon  \,,\\
 p_0^ \epsilon  &=  N_0^ \epsilon  \cdot \left[ \nu   \operatorname{Def} u_0^ \epsilon  \cdot N_0^ \epsilon  \right] \ \ &&\text{on} \ \ \Gamma^ \epsilon   \,,
\end{alignat}
\end{subequations}
so that using the same proof as that of Lemma \ref{lemma1}, we have the following $ \epsilon $-independent elliptic estimate:
\begin{equation}\label{p0-est}
\|p_0^ \epsilon \|_{2, \Omega^ \epsilon } \le C \left[ \|u^0 _\epsilon \|_{3,\Omega^ \epsilon}+ \|u^0_ \epsilon \|^2_{3,\Omega^ \epsilon } \right]  \le C \, 
\mathcal{P} ( m_0) \,,
\end{equation} 
where we use $\mathcal{P} $ denote denote a generic polynomial function that depends only on $\Omega$.

\section{A priori estimates}\label{section7}

Let $\Omega^ \epsilon  $ denote the dinosaur domain shown in Figure \ref{fig_dino}, and let $\theta_l$ denote the system of local charts for $\Omega^ \epsilon  $
as defined in (\ref{normalchart}).
By denoting $\eta_l = \eta \circ \theta_l$ we see that
$$
\eta_l(t): B^+ \to  \Omega(t) \ \text{ for } \ \ l=1,...,K \,.
$$
We set $v_l = u \circ \eta_l$, $q_l = p \circ \eta_l$ and $A_l= [ D \eta_l ] ^{-1} $, $J_l = C_l$ (where $C_l>0$ is a constant, and $a_l = J_l A_l$.  
 The unit normal
$n_l$ is defined as $ \mathfrak{g}  ^ {-\frac{1}{2}} \frac{\p \eta_l}{\p x_1} \times \frac{\p \eta_l}{\p x_2} $ if $d=3$ and by 
$\mathfrak{g}  ^ {-\frac{1}{2}}{ \frac{\p \eta_l}{\p x_1}}^ \perp$ if $d=2$. 

It follows that for $l=1,...,K$,
\begin{subequations}
\label{localNS}
\begin{alignat}{2}
\eta_l(t) &= \theta_l + \int_0^t v_l\ \ && \text{ in } B^+ \times [0,T] \,, \label{localNS.a0} \\
\p_t  v_l  + A_l^T \nabla q_l   &=  \Delta _ {\eta_l} v_l \ \ && \text{ in } B^+ \times (0,T] \,, \label{localNS.a} \\
\operatorname{div} _{\eta_l} v_l &=0 \ \ && \text{ in } B^+ \times [0,T] \,,\label{localNS.b}  \\
\nu \operatorname{Def} _ {\eta_l} v_l \cdot n_l - q_l \, n_l &=0 \ \ && \text{ on } B^0 \times [0,T] \,,\label{localNS.c}  \\
(\eta_l,v_l)  &=(\theta_l,u_0 \circ \theta_l ) \ \  \ \ && \text{ in } B^+ \times \{t=0\} \,, \label{localNS.e}
\end{alignat}
\end{subequations}
where we have set $\nu=1$.

\begin{definition}[Higher-order energy function]   For each $t\in[0,T]$, we define the higher-order energy function 
\begin{align*} 
E^\epsilon (t) & =  1+ \| \eta( \cdot ,t)\|_{3, \Omega^ \epsilon  }^2 +  \| v ( \cdot ,t)\|_{2, \Omega^ \epsilon  }^2  
+ \int_0^t \| v ( \cdot ,s)\|_{3, \Omega^ \epsilon  }^2 ds
%+  \| q^ \epsilon ( \cdot ,t)\|_{1, \Omega^ \epsilon  }^2 
+ \int_0^t \| q( \cdot ,s)\|_{2, \Omega^ \epsilon  }^2 ds  \\
& \qquad +   \| v_t( \cdot ,t)\|_{0, \Omega^ \epsilon  }^2  
+ \int_0^t \| v_t ( \cdot ,s)\|_{1, \Omega^ \epsilon  }^2 ds 
\end{align*} 
We then set  $M_0 = \mathcal{P} ( E^ \epsilon (0))$ where $\mathcal{P} $ denotes a generic polynomial whose coefficients depend only on $\Omega$.
The constant $M_0$ is then equal to $\mathcal{P} (m_0)$, a polynomial function of the constant $m_0$ introduced in (\ref{u0-est}).
\end{definition} 

\begin{theorem}\label{prop1}  Assuming that $\Gamma(t)$ does not self-intersect, independent of $ \epsilon  >0$, there exists a time $T>0$ and a constant $C>0$ such that the solution 
\begin{align*} 
%\eta &  \in C([0,T], H^3( \Omega^ \epsilon )) \,, \\
v \in C([0,T], H^2( \Omega^ \epsilon )) \cap L^2(0,T; H^3(\Omega^ \epsilon )) \,, \ \ 
q  \in L^2(0,T; H^2(\Omega^ \epsilon ))
\end{align*} 
 to
(\ref{NSlag}) satisfies the a priori estimate:
\begin{align}
 \max_{t\in [0,T]} E^ \epsilon (t)    \le C\,  M_0   \,.   \label{main-est}
\end{align}
\end{theorem}
\begin{proof}

The proof will proceed in five steps. 

\vspace{.1 in}
\noindent
{\bf Step 1. Estimates for  $ \nabla \eta$ and $A$.}  
Using (\ref{localNS}a), we see that
\begin{equation}\label{t-est}
\| \nabla \eta (\cdot , t)- \operatorname{Id} \|_{ 2, \Omega^ \epsilon } \le \left\| \int_0^t \nabla v( \cdot ,s)ds \right \|_{ 2, \Omega^ \epsilon } \le \sqrt{t} \sup_{s \in [0,t]} \sqrt{E^ \epsilon (t)}\,.
\end{equation} 
Thanks to Lemma \ref{lemma2}, there exists a constant $C>0$, independent of $ \epsilon $, such that
\begin{equation}\label{est-eta}
\| \nabla \eta (\cdot , t)- \operatorname{Id} \|_{ L^\infty(\Omega^ \epsilon )}  \le C \sqrt{t} \sup_{s \in [0,t]} \sqrt{E^ \epsilon (t)}\,.
\end{equation} 

 Since $\det \nabla \eta=1$, the matrix $A$ is simply the cofactor matrix of 
$ \nabla \eta$:
\begin{equation}\label{def-A}
A=
\left[
\begin{matrix}
- {\bf \eta},_2^\perp \\
 {\bf \eta},_1^\perp
\end{matrix}
\right] \text{ for } d=2, \text{ and } 
A=
\left[
\begin{matrix}
{\bf \eta,_2 \times \eta,_3}\\
{\bf \eta,_3 \times \eta,_1}\\
{\bf \eta,_1 \times \eta,_2}
\end{matrix}
\right] \text{ for } d=3\,,
\end{equation} 
where each row is a vector, and for a $2$-vector $x=(x_1,x_2)$, $x^\perp=(-x_2, x_1)$.

We make the following basic assumption, that we shall verify below in Step 5: for a constant $ 0 < \vartheta \ll 1$,  we suppose that $t \in[0,T]$ and that
$T$ is chosen sufficiently small so that
\begin{equation}\label{basic}
\sup_{t \in [0,T]} \| \nabla \eta (\cdot , t)- \operatorname{Id} \|_{ L^\infty(\Omega^ \epsilon )} \le \vartheta^{10} \,.
\end{equation} 
It follows from (\ref{def-A}), that since $\|A ( \cdot , t) - \operatorname{Id} \|_{L^\infty (\Omega^ \epsilon )}  \le \int_0^t \|A_t ( \cdot ,s) \|_{L^\infty (\Omega^ \epsilon )} ds$,
\begin{equation}\label{est-A1}
\sup_{t \in [0,T]} \|A ( \cdot , t) - \operatorname{Id} \|_{L^\infty (\Omega^ \epsilon )} +   \|A A^T( \cdot , t) - \operatorname{Id} \|_{L^\infty (\Omega^ \epsilon )} 
%+   \|A A^T \times A ( \cdot , t) - \operatorname{Id}\times \operatorname{Id} \|_{L^\infty (\Omega^ \epsilon )} 
\le  \vartheta\,.
\end{equation} 

\vspace{.1 in}
\noindent
{\bf Step 2. Boundary regularity.}
 We begin by considering a single boundary  chart $\theta_l : B^+ \to \Omega(t)$.   Let $\zeta_l$ denote the smooth cut-off function defined in Section \ref{sec::partition}.
Using equation (\ref{localNS}b), 
we compute  the following $L^2(B^+)$ inner-product:
\begin{equation}\label{cs0}
 \left( \zeta_l \bar \p^2 [  \p_t v_l -  \Delta _\eta v + A_l^T \,  \nabla q_l ] \ , \ \zeta_l \bar \p^2 v_l \right)_{L^2(B^+)} =0 \,.
\end{equation} 
To simplify the notation, we fix $l \in \{1,...,K\}$ and drop the subscript.  The chart $\theta_l$ was defined so that $\det \nabla \theta_l=C_l$ for a constant
$C_l>0$. Then (\ref{cs0}) can be written as
be written as
\begin{equation}\label{cs8}
\int_{B^+} \zeta ^2 \bar \p^2 v_t^i \,  \bar\p^2 v^i \, dx - \int_{B^+}  \zeta ^2 \bar\p^2 [ A^k_s A^j_s v ^i,_j],_k \, \bar\p^2 v^i \, dx + \int_{B_+} \zeta ^2 \bar\p^2
[A^k_i q],_k \, \bar \p^2 v^i \, dx =0 \,.
\end{equation}

Integration-by-parts with respect to $x_k$ shows that
\begin{align} \label{cs1}
0 =  {\frac{1}{2}} \frac{d}{dt} \| \zeta \bar \p^2 v(t)\|^2_{0,B^+}
+  \int_{B^+}  \bar \p^2 [ A^k_s A^j_s v^i,_j]\, \bar\p^2[ \zeta ^2 v^i],_k dx 
+  \int_{B^+}  \bar \p^2 [A^k_i  q]\, \bar\p^2[ \zeta ^2 v^i],_k dx 
\end{align}
where we have used the boundary condition (\ref{localNS}d) to show that the boundary integral vanishes.   Using $ \delta ^{jk}$ to denote the
Kronecker delta function,  we write (\ref{cs1}) as
\begin{align}
 &{\frac{1}{2}} \frac{d}{dt} \| \zeta \bar \p^2 v(\cdot , t)\|^2_{0,B^+}
+  \| \zeta \bar\p^2  \nabla v(t)  \|^2_{0,B^+}
= -   \int_{B^+}  \bar \p^2 [A^k_i  q]\, \bar\p^2[ \zeta ^2 v^i],_k dx \nonumber \\
& \qquad 
-  \int_{B^+}  \bar \p^2 [ (A^k_s A^j_s - \delta^{kj}) v^i,_j]\, \bar\p^2[ \zeta ^2 v^i],_k dx -  \int_{B^+} \left[ \bar \p^2  v^i,_k\, (\bar\p^2 \zeta ^2 v^i 
+ 2 \bar \p \zeta^2 \bar \p v^i),_k  + \xi,_k \bar \p^2 v^i \right] dx \,.  \label{cs2}
\end{align}
We integrate (\ref{cs2}) over the time interval $[0,T]$:
\begin{align}
 {\frac{1}{2}}  \| \zeta \bar \p^2 v(\cdot , t)\|^2_{0,B^+} + \int_0^T   \| \zeta \bar\p^2   v(t)  \|^2_{1,B^+} \le  M_0 %+  T P( \sup_{t \in [0,T]} E^ \epsilon (t)) 
 +\mathcal{I} _1
 + \mathcal{I} _2 +\mathcal{I} _3 \label{cs7}
\end{align}
where
\begin{align*} 
\mathcal{I} _1 & =  \int_0^T \int_{B^+} \left| \bar \p^2 [A^k_i  q]\, \bar\p^2[ \zeta ^2 v^i],_k\right| dx dt \,, \\
\mathcal{I} _2 & =\int_0^T  \int_{B^+} \left|  \bar \p^2 [ (A^k_s A^j_s - \delta^{kj}) v^i,_j]\, \bar\p^2[ \zeta ^2 v^i],_k \right| dxdt\,, \\
\mathcal{I} _3 & =\int_0^T  \int_{B^+} \left|  \bar \p^2  v^i,_k\, [ \bar\p^2 \zeta ^2 v^i  + 2 \bar \p \zeta^2 \bar \p v^i],_k  + \xi,_k \bar \p^2 v^i  \right| dxdt  \,.
\end{align*} 
Using the Sobolev embedding theorem and Lemma \ref{lemma2}
We estimate $ \mathcal{I} _1$ 
\begin{align*} 
\mathcal{I} _1 & \le  \underbrace{\int_0^T \int_{B^+} | \bar \p^2 q| \, | A^k_i \bar \p^2 v^i,_k| \, dxdt}_{\mathcal{I} _1^a} 
 +  \underbrace{ \int_0^T  \| q\|_{2, \epsilon }   \| A\|_{2, \Omega^ \epsilon} \|v\|_{2, \Omega ^ \epsilon } dt}_{ \mathcal{I} _1^b}  \\
 & \qquad \qquad
+   \underbrace{\int_0^T  \| q\|_{1.5, \epsilon }   \| A\|_{2, \Omega^ \epsilon} \|v\|_{3, \Omega ^ \epsilon } dt}_{ \mathcal{I} _1^c} \,.
\end{align*} 
To estimate the integral $ \mathcal{I} _1^a$, we use (\ref{localNS}c) to write 
$$
v^i,_{ k \alpha \beta } A^k_i = - A^k_i,_{ \alpha  \beta } v^i,_k - A^k_i,_ \beta v^i,_{k \alpha } - A^k_i,_ \alpha  v^i,_{k \beta  } \,,
$$
so that the term with three derivatives on $v$ is converted to a term with three derivatives on $\eta$ plus lower-order terms.
It follows that for $ \delta >0$, and a constant $C_ \delta $ (which blows-up as $ \delta \to 0$), 
$$
 \mathcal{I} _1^a \le \delta   \int_0^T \|q\|^2_{2, \Omega ^ \epsilon } dt + C_ \delta  T P( \sup_{t \in [0,T]} E^ \epsilon (t)) \,.
$$
The integral $ \mathcal{I} _1^b$ is estimated in the same way.  For the integral $ \mathcal{I} _1^c$ we use linear interpolation to estimate the norm
$\int_0^T  \| q\|_{1.5, \epsilon }$:
\begin{align*} 
\mathcal{I} _1^c & \le  \delta  \int_0^T \|v\|^2_{3, \Omega ^ \epsilon } dt 
 +  \delta  \int_0^T \|q\|^2_{2, \Omega ^ \epsilon } dt +
   C_ \delta  T P( \sup_{t \in [0,T]} E^ \epsilon (t)) \,.
\end{align*} 
It follows that
\begin{equation}\label{cs3}
\mathcal{I} _1 \le  M_0 +  C_ \delta  T P( \sup_{t \in [0,T]} E^ \epsilon (t)) + \delta   \sup_{t \in [0,T]} E^ \epsilon (t) \,.
\end{equation} 

Next, for the integral $ \mathcal{I} _2$, 
\begin{align*} 
\mathcal{I} _2 & \le  \underbrace{\int_0^T  \int_{B^+} \left|    (A^k_s A^j_s - \delta^{kj}) \bar \p^2 v^i,_j\, \bar\p^2[ \zeta ^2 v^i],_k \right| dxdt}_{ \mathcal{I} _2^a}
+ \underbrace{2 \int_0^T  \int_{B^+} \left|  \bar \p  (A^k_s A^j_s - \delta^{kj}) \bar \p v^i,_j\, \bar\p^2[ \zeta ^2 v^i],_k \right| dxdt}_{ \mathcal{I} _2^b} \\
& \qquad + \underbrace{\int_0^T  \int_{B^+} \left|  \bar \p^2  (A^k_s A^j_s - \delta^{kj})  v^i,_j\, \bar\p^2[ \zeta ^2 v^i],_k \right| dxdt }_{ \mathcal{I} _2^c} \,.
\end{align*} 
Using (\ref{est-A1}) and choosing $\vartheta < \delta $, 
$$
\mathcal{I} _2^a \le C_ \delta  T P( \sup_{t \in [0,T]} E^ \epsilon (t)) + \delta   \sup_{t \in [0,T]} E^ \epsilon (t) \,.
$$
In the same way as above, we again use Lemma \ref{lemma2}, together with linear interpolation for term $ \mathcal{I} _2^b$, to see that 
\begin{equation}\label{cs4}
\mathcal{I} _2 \le  M_0 +  C_ \delta  T P( \sup_{t \in [0,T]} E^ \epsilon (t)) + \delta   \sup_{t \in [0,T]} E^ \epsilon (t) \,.
\end{equation} 

The integral  $ \mathcal{I} _3$ is straightforward and also satisfies
\begin{equation}\label{cs5}
\mathcal{I} _3 \le  M_0 +  C_ \delta  T P( \sup_{t \in [0,T]} E^ \epsilon (t)) + C\delta   \sup_{t \in [0,T]} E^ \epsilon (t) \,.
\end{equation}

Summing over all of the  boundary charts $l=1,...,K$ in (\ref{cs7}), the inequalities (\ref{cs3})--(\ref{cs5}) together with the trace theorem, Lemma 
\ref{lemma3},  show that
\begin{equation}\label{cs6}
 \int_0^T   \|   v(\cdot , t)  \|^2_{2.5, \Gamma^ \epsilon } \le 
M_0 +  C_ \delta  T P( \sup_{t \in [0,T]} E^ \epsilon (t)) + \delta   \sup_{t \in [0,T]} E^ \epsilon (t)
\end{equation} 

\noindent
{\bf Step 3. Estimates for the time-differentiated problem.} We consider the time-differentiated version of (\ref{NSlag}) which we write as the
following system:
\begin{subequations}
\label{NSlagt}
\begin{alignat}{2}
\eta_t  & = v  \ && \text{ in } \Omega^ \epsilon  \times [0,T] \,,  \\
 v_{tt}   - \Delta _ \eta v_t + A^T \nabla q_t   &= - A^T_t \nabla q  + [ \p_t( A^j_s A^k_s) v,_k],_j \ \ && \text{ in } \Omega^ \epsilon  \times (0,T] \,,  \\
\operatorname{div} _\eta v_t &=-v^i,_j \p_t A^j_i \ \ && \text{ in } \Omega^\epsilon  \times [0,T] \,,  \\
\p_t\left[ \operatorname{Def} _ \eta v \cdot n - qn\right] &=0\ \ && \text{ on } \Gamma^ \epsilon  \times [0,T] \,,  \\
(\eta,v,v_t)  &=(e,u_0^ \epsilon ,u_1^ \epsilon ) \ \  \ \ && \text{ in } \Omega^\epsilon  \times \{t=0\} \,, 
\end{alignat}
\end{subequations}
where $u_1^ \epsilon = \Delta u_0^ \epsilon - \nabla p_0^ \epsilon $, with $u_0^ \epsilon $ defined in (\ref{Stokes2}) and $p_0^ \epsilon $ 
defined in (\ref{p0}); therefore,  independently of $ \epsilon >0$,
\begin{equation}\label{u1}
\| u_1^ \epsilon \|_{0, \Omega^\epsilon  } \le \mathcal{P} (m_0) \,.
\end{equation} 
We define the space of $ \operatorname{div} _\eta$-free vectors fields on $\Omega^\epsilon $ as
$$
\mathcal{V} (t) = \{ \phi  \in H^1(\Omega^\epsilon ; \mathbb{R}  ^d) \ : \ \operatorname{div} _{\eta( \cdot , t)} \phi =0 \}\,.
$$
Taking the $ L^2(\Omega^ \epsilon ) $ inner-product of equation (\ref{NSlagt}b) with a test function $\phi \in \mathcal{V} (t)$, we have that
\begin{equation}\label{weak}
\int_{ \Omega^\epsilon } v_{tt} \cdot \phi dx + \int_{\Omega ^ \epsilon } \p_t [A^k_s A^j_s v^i,_j] \, \phi^i,_k dx =\textcolor{black}{ \int_ \Omega q\, \p_t A^k_i \phi^i,_k \, dx } \ \ \forall \phi \in \mathcal{V} (t) \,.
\end{equation} 
Next, we define a vector field $w$ satisfying 
\begin{subequations}
  \label{w}
\begin{alignat}{2}
\operatorname{div} _ \eta  w  &= -v^i,_j \p_t A^j_i    \ \  \ &&\text{in} \ \ \Omega^ \epsilon  \,,\\
w  &=  \phi(t) n  \ \ &&\text{on} \ \ \Gamma^ \epsilon   \,,
\end{alignat}
\end{subequations}
where $\phi(t) = -\int_{ \Omega^\epsilon } -v^i,_j \p_t A^j_i dx / | \Gamma^ \epsilon |$.
A solution $w$ can be found by solving a Stokes-type problem, and according to the proof of Lemma 3.2 in \cite{ChSh2010}, for integers $k\ge 1$,
\begin{equation}\label{est-w}
\| w( \cdot , t) \|_{k, \Omega^\epsilon } \le C\left( \|  v^i,_j (\cdot , t) \,  \p_t A^j_i(\cdot , t) \|_{k-1, \Omega^\epsilon } +  \| \phi(t) n \|_{k-1/2, \Gamma^\epsilon }\right) \,,
\end{equation} 
where the constant $C$ is independent of $ \epsilon $ by Lemma \ref{lemma1}.
It follows from (\ref{est-w}) and (\ref{def-A}) that
\begin{equation}\label{est-w2}
\sup_{t \in [0,T]} \| w( \cdot , t) \|_{2, \Omega^\epsilon } + \int_0^T \| w( \cdot , t) \|_{3, \Omega^\epsilon } ^2 \le T P( \sup_{t \in [0,T]} E^ \epsilon (t)) \,.
\end{equation} 

   Similarly, 
\begin{subequations}
  \label{wt}
\begin{alignat}{2}
\operatorname{div} _ \eta  w_t  &=  -\left(w^i,_j \p_t A^j_i + \p_t(v^i,_j \p_t A^j_i )\right)   \ \  \ &&\text{in} \ \ \Omega^ \epsilon  \,,\\
w_t  &=  \left( \phi_t n \right) _t \ \ &&\text{on} \ \ \Gamma^ \epsilon   \,.
\end{alignat}
\end{subequations}
and
\begin{equation}\nonumber
\| w_t\|_{1, \Omega^\epsilon } \le C\left( \| w^i,_j \p_t A^j_i + \p_t(v^i,_j \p_t A^j_i)\|_{0, \Omega^\epsilon }  +
\| (\phi_t n)_t\|_{1/2, \Gamma^\epsilon } \right) \,,
\end{equation} 
so that
\begin{equation}\label{est-wt}
\int_0^T \| w_t\|_{1, \Omega^\epsilon }^2 \le T P( \sup_{t \in [0,T]} E^ \epsilon (t)) \,.
\end{equation} 
Now, because of (\ref{w}a),
$v_t -w \in \mathcal{V} (t)$, and we are allowed to set $\phi = v_t -w$ in (\ref{weak}).   We find that
\begin{align*} 
{\frac{1}{2}} \frac{d}{dt} \| v_t ( \cdot ,t) \|^2_{0, \Omega^\epsilon }+ \int_{\Omega ^ \epsilon } \p_t [A^k_s A^j_s v^i,_j]\, v_t^i,_k dx
&= \int_{ \Omega^\epsilon } v_{tt} \cdot w dx + \int_{\Omega^\epsilon }  \p_t (A^k_s A^j_s v^i,_j)  \, w^i,_k dx \\
& \qquad \qquad +  \int_ \Omega q\, \p_t A^k_i  \left[v_t^i,_k + w^i,_k \right] \, dx  \,.
\end{align*} 
and hence for $t \in (0,T)$,
\begin{align*} 
& {\frac{1}{2}}  \| v_t ( \cdot ,t) \|^2_{0, \Omega^\epsilon }+ \int_0^t\| \nabla v_t\|^2_{0, \Omega^\epsilon } ds = {\frac{1}{2}}  \|u_1 \|^2_{0, \Omega^\epsilon }
\overbrace{- \int_0^t \int_{\Omega ^ \epsilon }
 [A^k_s A^j_s  - \delta ^{kj}] v_t^i,_j\, v_t^i,_k dx ds}^{ \mathcal{J} _1}\\
& \qquad   \underbrace{- \int_0^t \int_{\Omega ^ \epsilon } \p_t [A^k_s A^j_s ] v^i,_j v_t^i,_k dxds}_{ \mathcal{J} _2}
+ \underbrace{\int_0^t \int_{ \Omega^\epsilon } v_{tt} \cdot w dxds}_{ \mathcal{J} _3}
+ \underbrace{ \int_0^t \int_{\Omega^\epsilon } \p_t [A^k_s A^j_s v^i,_j]\, w^i,_k dxds}_{ \mathcal{J} _4}\\
& \qquad  
+ \underbrace{ \int_0^t \int_ \Omega q\, \p_t A^k_i  \left[v_t^i,_k + w^i,_k \right] \, dxds}_{ \mathcal{J} _5} \,.
\end{align*} 
For $  \delta >0$ and using (\ref{est-A1}) with $\vartheta< \delta $, we see that
\begin{equation}\label{csj1}
| \mathcal{J} _1| \le \delta \sup_{t \in [0,T]} E^ \epsilon (t) \,.
\end{equation} 
Next, according to (\ref{def-A}) the components of $A$ are either linear ($d=2$) or quadratic ($d=3$) with respect to the components of $ \nabla \eta$;
hence, $\p_t A $ behaves like $ \nabla v$ for $d=2$ and like $ \nabla \eta \, \nabla v$ for $d=3$.   We consider the more difficult case that $d=3$ in which
case $\p_t (A A^T)$ behaves like $ \nabla \eta \, \nabla \eta \, \nabla \eta \, \nabla v$.   It follows by the Cauchy-Young inequality that for $ \delta >0$, we have that
\begin{equation}\label{csj2}
| \mathcal{J}_2| \le 
M_0 +     T \mathcal{P} (\sup_{t \in [0,T]}  E^ \epsilon (t)) +
\delta \sup_{t \in [0,T]} E^ \epsilon (t) \,.
\end{equation} 

To estimate $\mathcal{J} _3$, we integrate-by-parts in time:
\begin{align} 
|\mathcal{J} _3| & \le  \int_0^t \int_{ \Omega^\epsilon } |v_{t} \cdot w_t | dxds + \left| \left. \int_{ \Omega^\epsilon } v_{t} \cdot w  dx\right|^t_0 \right| 
\nonumber \\
& \le  \int_0^t \int_{ \Omega^\epsilon } |v_{t} \cdot w_t | dxds + M_0 +  \int_{ \Omega^\epsilon } | v_t ( \cdot ,t) w( \cdot ,0)| dx
+ \int_{ \Omega^\epsilon } \left| v_t ( \cdot ,t) \int_0^t w_t(\cdot ,s)ds\right| dx \nonumber \\
& \le M_0 +     T \mathcal{P} (\sup_{t \in [0,T]}  E^ \epsilon (t)) +
\delta \sup_{t \in [0,T]} E^ \epsilon (t) \,, \label{csj3}
\end{align} 
the last inequality following from the Cauchy-Young inequality and the estimates (\ref{est-w}) and (\ref{est-wt}).   
The integrals $ \mathcal{J} _4$ and $ \mathcal{J} _5$ (using (\ref{est-w2}) and (\ref{est-wt})) are
estimated in the same way as $ \mathcal{J} _2$ so that 
\begin{equation}\label{csj4}
| \mathcal{J}_4| + | \mathcal{J}_5| \le 
M_0 +     T \mathcal{P} (\sup_{t \in [0,T]}  E^ \epsilon (t)) +
\delta \sup_{t \in [0,T]} E^ \epsilon (t) \,.
\end{equation} 
Combining the estimates (\ref{csj1})--(\ref{csj4}), we find that
\begin{equation}\label{cs10}
\sup_{t \in [0,T]}   \| v_t ( \cdot ,t) \|^2_{0, \Omega^\epsilon }+ \int_0^T\| v_t\|^2_{1, \Omega^\epsilon } dt 
 \le 
M_0 +     T \mathcal{P} (\sup_{t \in [0,T]}  E^ \epsilon (t)) +
C\delta \sup_{t \in [0,T]} E^ \epsilon (t) \,.
\end{equation}

\vspace{.1 in}
\noindent
{\bf Step 4. Regularity for the velocity and pressure.}  
Next, we write equation (\ref{NSlag}b) as
\begin{subequations}
\label{stokes_for_v}
\begin{alignat}{2}
-\Delta v + \nabla q & = \operatorname{div} [ (AA^T - \operatorname{Id} ) \nabla v]  - (A^T -\operatorname{Id} ) \nabla q-v_t \ \ && \text{ in } \Omega^ \epsilon  \times (0,T] \,, \label{stokes_for_v.a} \\
\operatorname{div}  v &= - (A^j_i - \delta ^j_i) v^i,_j \ \ && \text{ in } \Omega^\epsilon  \times [0,T] \,,  \\
v & \in L^2(0,T; H^{2.5}(\Gamma^ \epsilon ) &&
\end{alignat}
\end{subequations}

The two inequalities (\ref{cs6}) and (\ref{cs10}), together with the Stokes regularity given in  Lemma \ref{lemma1}, show that 
$v \in L^\infty([0,T]; H^2(\Omega^ \epsilon ) ) \cap L^2(0,T; H^3( \Omega^\epsilon ))$ and satisfies
\begin{equation}\label{cs11}
\sup_{t \in [0,T]}   \| v ( \cdot ,t) \|^2_{2, \Omega^\epsilon }+ \int_0^T\| v\|^2_{3, \Omega^\epsilon } dt +  \int_0^T\| q\|^2_{2, \Omega^\epsilon } dt 
 \le 
M_0 +     T \mathcal{P} (\sup_{t \in [0,T]}  E^ \epsilon (t)) +
C\delta \sup_{t \in [0,T]} E^ \epsilon (t) \,.
\end{equation} 
By choosing $ \delta >0$ sufficiently small, we obtain that 
\begin{equation}
\label{cs12}
\sup_{t \in [0,T]} E^ \epsilon (t)
 \le 
M_0 +     T \mathcal{P} (\sup_{t \in [0,T]}  E^ \epsilon (t))\,,
\end{equation}
for a constant $M_0$ and a polynomial function $ \mathcal{P} $ which are both independent of $ \epsilon $.  

From the estimate (\ref{cs11}),  $v\in L^2(0,T; H^3(\Omega^ \epsilon ) )$,  and the estimate (\ref{cs10}),   $v_t\in L^2(0,T; H^1(\Omega^ \epsilon ) )$.
Using the partition of unity functions $\zeta_l$ defined in Step 2 above, we then see that for each chart
$\zeta_l v\in L^2(0,T; H^3(\mathcal{B}_l ) )$ where $ \mathcal{B}_l = B^+ $ for $l=1,...,K$, and $ \mathcal{B}_l = B$ for $l=K+1,...,L$.  Similarly,
$\zeta_l v_t\in L^2(0,T; H^1(\mathcal{B}_l ) )$.   It is then standard that 
$\zeta_l v\in C^0([0,T]; H^2( \mathcal{B} _l ) )$, and hence by summing over $l=1,...,L$, $ v\in C^0([0,T]; H^2( \Omega^\epsilon  ) )$.

Since the pressure satisfies the elliptic system:
\begin{alignat*}{2}
-\Delta _\eta q & = v^i,_rA^r_j v^j,_sA^s_i\ \ && \text{ in } \Omega^ \epsilon  \times (0,T] \,,  \\
q &=  n\cdot \left[ \text{Def}_\eta v\cdot n\right] \ \ && \text{ on } \Gamma^\epsilon  \times [0,T] \,,  
\end{alignat*}
we then infer that $q\in C^0([0,T]; H^1(\Omega^ \epsilon ) )$.  Then, using the momentum equation
 (\ref{stokes_for_v.a}), it follows that $v_t\in C^0([0,T]; L^2(\Omega^ \epsilon ) )$.

%\textcolor{red}
{  This then shows that $E^ \epsilon(t)$ is a  continuous function of time}. Following Section 9 in \cite{CoSh2006}, from (\ref{cs12}), we now may choose $T>0$ sufficiently small and independent of $ \epsilon $,  such that
\begin{equation}\label{apriori-est}
\sup_{t \in [0,T]} E^ \epsilon (t) \le 2M_0 \,.
\end{equation} 
 
 \vspace{.1 in}
\noindent
{\bf Step 5. Verifying the basic assumption (\ref{basic}).}  Having established (\ref{apriori-est}) on $[0,T]$ with $T$ independent of $ \epsilon $, for
any $\varepsilon>0$,  we may now use the formula (\ref{t-est}) to choose $T$ even smaller if necessary to ensure that (\ref{basic}) holds.   This concludes
the proof.
\end{proof} 

We now establish a more quantitative estimate in order to assess the continuity of $\bar \p^2 v(t , \cdot )$ in  $ L^2( \Omega^\epsilon ) $.

\begin{proposition} \label{prop2} For all $t\in [0,T]$, 
\begin{equation}\label{est-main2}
\max_{x \in \Omega^ \epsilon } \| \bar \p^2 ( v^ \epsilon ( \cdot ,t)  - u_0^ \epsilon )\|_{0, \Omega ^ \epsilon }^2 
+ \int_0^t \| \bar \p^2 ( v^ \epsilon(\cdot ,s)  - u_0^ \epsilon )\|_{1, \Omega ^ \epsilon }^2 ds \lesssim t^{1/2} \mathcal{P} (M_0) \,.
\end{equation} 
\end{proposition}
\begin{proof} 
We write $v(t) = v( \cdot ,t)$ and again set viscosity $\nu=1$. The difference $v(t) - \uu$ satisfies the equation 
\begin{equation}\nonumber
(v-\uu)_t - \Delta _ \eta( v- \uu) + A^T \nabla q = \Delta _ \eta \uu \,.
\end{equation} 
Following Step 2 in the proof of Theorem \ref{prop1}, and once again localize to a boundary chart $\theta_l$, $l=1,...,K$, with $ 
\det \nabla \theta_l=C_l$ and with cut-off functions $\zeta_l$, we obtain that
\begin{align}
0 &=  {\frac{1}{2}} \frac{d}{dt} \| \zeta \bar \p^2 [v(t)-\uu]\|^2_{0,B^+}
+  \int_{B^+}  \bar \p^2 [ A^k_s A^j_s (v-\uu),_j] \cdot  \bar\p^2[ \zeta ^2 (v-\uu)],_k dx  \nonumber \\
&\qquad \qquad\qquad
+  \int_{B^+}  \bar \p^2 [A^k_i  q]\, \bar\p^2[ \zeta ^2 v^i],_k dx  +  \int_{B^+}  \bar \p^2 [ A^k_s A^j_s {\uu},_j] \cdot  \bar\p^2[ \zeta ^2 (v-\uu)],_k dx\,,
 \label{cs101}
\end{align}
where we have dropped the explicit chart dependence on $l$ and where again,
the boundary integral terms  have vanished due to (\ref{NSlag}d).
We integrate (\ref{cs101}) over the time interval $[0,T]$:
\begin{align*}
  \| \zeta \bar \p^2 [v(t)- \uu]\|^2_{0,B^+} + \int_0^T   \| \zeta \bar\p^2   [v(t)-\uu]  \|^2_{1,B^+} \le |\mathcal{K} _1|
 + |\mathcal{K} _2| + |\mathcal{K} _3| + |\mathcal{K} _4| \,,
\end{align*}
where we are writing $\uu$ for $\uu \circ \theta_l$, and where
\begin{align*} 
\mathcal{K} _1 & =   \int_0^T \int_{B^+} \bar \p^2 [A^k_i  q]\, \bar\p^2[ \zeta ^2 (v-\uu)^i],_k dx dt \,, \\
\mathcal{K} _2 & =\int_0^T  \int_{B^+}   \bar \p^2 [ (A^k_s A^j_s - \delta^{kj}) (v-\uu),_j] \cdot  \bar\p^2[ \zeta ^2 (v-\uu)],_k  dxdt\,, \\
\mathcal{K} _3 & =\int_0^T  \int_{B^+}   \bar \p^2  (v-\uu)^i,_k\, [\ [ \bar\p^2 \zeta ^2 (v-\uu)^i  + 2 \bar \p \zeta^2 \bar \p (v-\uu)^i],_k\textcolor{black} {+\zeta^2,_k \bar\p^2 v^i]}  dxdt  \,, \\
\mathcal{K} _4 & =\int_0^T  \int_{B^+}  \bar \p^2 [ (A^k_s A^j_s  {\uu},_j] \cdot  \bar\p^2[ \zeta ^2 (v-\uu)],_k  dxdt \,.
\end{align*} 
We write
\begin{align*} 
\mathcal{K} _1\le   \underbrace{ \int_0^T \int_{B^+} \bar \p^2 [A^k_i  q]\, \bar\p^2[ \zeta ^2 v^i],_k dx dt }_{\mathcal{K} _1^a}  + 
\underbrace{\int_0^T \int_{B^+} \left| \bar \p^2 [A^k_i  q]\, \bar\p^2[ \zeta ^2 {\uu}^i],_k\right| dx dt }_{ \mathcal{K} _1^b} \,.
\end{align*} 
By (\ref{u0-est}) and (\ref{main-est}), we see that
$$
| \mathcal{K} _1^b | \le \sqrt{T} \mathcal{P} (M_0) \,.
$$
For the integral $ \mathcal{K} _1^a$, we focus on the integrand that arises when $\bar \p^2$ acts on {\it both} $q$ and $v^i,_k$, for all other derivative
combinations immediately give an integral bound of $ \sqrt{T} \mathcal{P} (M_0)$.   Using the Lagrangian divergence-free condition (\ref{NSlag}c),
\begin{align*} 
\left| \int_0^T \int_{B^+} \zeta^2 \bar \p^2q \, A^k_i \bar \p^2 v^i,_k dx dt \right| & \le 
\left| \int_0^T \int_{B^+} \zeta^2 \bar \p^2q \, \bar \p^2 A^k_i  v^i,_k dx dt\right| \\
& \qquad \qquad \qquad 
+ 2\left| \int_0^T \int_{B^+} \zeta^2 \bar \p^2q \, \bar \p A^k_i  \bar \p v^i,_k dx dt\right| \,.
\end{align*} 
An application of the Cauchy-Young inequality together with the Sobolev embedding theorem,  shows that
$$
| \mathcal{K} _1^a | \le \sqrt{T} \mathcal{P} (M_0) \,.
$$
For the integral $ \mathcal{K} _2$, we consider the case that $ \bar \p^2$ acts on $(A^k_s A^j_s - \delta^{kj}) $, all other terms immediately giv\textcolor{red}{ing} the
desired bound.   Using (\ref{t-est}) and (\ref{def-A}),  $ \| A A^T - \operatorname{Id} \|_{ L^\infty(B^+)} \le \sqrt{T}  \mathcal{P} (M)$, so that with (\ref{main-est}),
$$
| \mathcal{K} _2 | \le \sqrt{T} \mathcal{P} (M_0) \,.
$$
The integral $ \mathcal{K} _3$ and $\mathcal{K} _4$ are easily estimated using the Cauchy-Young inequality, the Sobolev embedding theorem, and
(\ref{main-est}).      We have thus established that
$$
  \| \zeta \bar \p^2 [v_l(t)- \uu]\|^2_{0,B^+} + \int_0^T   \| \zeta \bar\p^2   [v_l(t)- \uu \circ \theta_l]  \|^2_{1,B^+} \le \sqrt{T} \mathcal{P} ( M_0) \,.
$$
Summing over  $l =1,...,K$ then concludes the proof.
\end{proof}

\section{Proof of the Main Theorem}\label{section8}
Using the Lagrangian divergence condition (\ref{NSlag}c), we have that $ \operatorname{div} v = -(A^j_i \textcolor{black}{-}\delta ^j_i) v^i,_j$, which we write as
$ \operatorname{div} v = \textcolor{black}{-}(A- \operatorname{Id} ) : \nabla v$.   Then, since $ \operatorname{div} \uu =$,  for all $t \in [0,T]$, 
\begin{equation}\label{div-est}
\| \bar \p \operatorname{div}  (v -\uu) \|^2_{0, \Omega^\epsilon } \le \| \bar \p (A- \operatorname{Id} ) \, \nabla v \|^2_{0, \Omega^\epsilon }
+    \| (A- \operatorname{Id} ) \, \bar \p \nabla v \|^2_{0, \Omega^\epsilon } \le \sqrt{T} \mathcal{P} (M_0) \,.
\end{equation} 
Using (\ref{div-est}) together with (\ref{est-main2}), the normal trace theorem (see, for example,  (A.6) in \cite{CoSh2014a}) shows that 
$\bar \p^2 (v - \uu) \cdot N_ \epsilon   \in C([0,T; H^{ - {\frac{1}{2}} }(\Gamma^ \epsilon ))$ and
$$
 \| \bar \p^2 (v - \uu) \cdot N_ \epsilon\|^2_{ -1/2, \Gamma^ \epsilon } \le \sqrt{T} \mathcal{P} (M_0) \,,
$$
so that
\begin{equation}\nonumber
 \| (v - \uu) \cdot N_ \epsilon\|^2_{ 1.5, \Gamma^ \epsilon } \le \sqrt{T} \mathcal{P} (M_0) \,,
\end{equation} 
and hence by Lemma \ref{lemma2},  
\begin{equation}\label{cs200}
\max_{x \in \Gamma^ \epsilon } \| (v(x,t) - \uu) \cdot N_ \epsilon\| \le T ^{\frac{1}{4}}  \mathcal{P} (M_0) \ \ \forall t \in[0,T] \,.
\end{equation} 

Next, we consider the motion of the points $X_+^ \epsilon $ and $X_-$ given in Section \ref{sec::u0} (see Figure \ref{fig_initialconditions}).  Recall
that the unit normal $N_ \epsilon $ at both the points $X_+^ \epsilon = (0,0, \epsilon ) $ and $X_-=(0,0,0)$ is vertical, so by definition of $\uu$, we have that
$$
\uu ( X_+^ \epsilon ) \cdot N_ \epsilon = -1 \, \ \ \uu ( X_-)\cdot N_ \epsilon =0 \,, \ \ \text{ and } | X_+^ \epsilon - X_-|= \epsilon \,.
$$
Using Theorem \ref{prop1}, we 
 choose $\epsilon $ so small that $ 10 \epsilon < T$, where $[0,T]$ is the time interval of existence which is independent of $ \epsilon $, and we
consider the vertical displacement of the falling particle $X_+^ \epsilon $.   Since $X_+^ \epsilon \cdot e_d  = \epsilon $, and 
$$\eta(X_+^ \epsilon ,t)  \cdot e_d  = \epsilon + \int_0^t  v^d(X_+^ \epsilon , s)ds\,,$$
for $t= 10 \epsilon $, we have from (\ref{cs200}) that
$$
\eta^d(X_+^ \epsilon , 10 \epsilon )   < -8 \epsilon \,.
$$
Next, let $Z$ denote any point on $\partial \omega_- \cap \{x_d=0\}$.  Since $\uu (Z) \cdot N_ \epsilon =0$ and 
$ \eta(Z, 10 \epsilon ) = \int_0^{10\epsilon} v(Z,s)ds$, 
according to (\ref{cs200}),
$$
\eta(Z, 10 \epsilon ) \cdot e_d  \ge -c \epsilon^ {\frac{5}{4}}  \,, \ \ c= 10^ {\frac{5}{4}}  \mathcal{P} (M_0) \,.
$$
We then choose $ \epsilon >0$ sufficiently small so that $c \epsilon^ {\frac{5}{4}}  < 8\epsilon  $.    It follows that 
\begin{equation}
\label{cs201}
\eta(X_+^ \epsilon , 10 \epsilon ) \cdot e_d < \eta(Z, 10 \epsilon ) \cdot e_d \,.
\end{equation}

We next consider the horizontal displacement of the particle $X_+^\epsilon$ and 
any particle $Z$ on $\partial \omega_- \cap \{x_d=0\} \times [0,10\epsilon]$.
From the  estimate (\ref{apriori-est}),  for all time $t\in [0,10\epsilon]$, $\|v(\cdot,t)\|_{L^\infty(\Omega)}\le \mathcal{P}(M_0)$.

Therefore, for any $t\in [0,10\epsilon]$ and for $ \alpha =1,...,d-1$,  
$$|\eta^ \alpha (X_+^ \epsilon , t )| \le10 \epsilon \mathcal{P}(M_0) \text{ and }
| \eta^ \alpha (Z , t ) -Z^ \alpha | \le 10 \epsilon \mathcal{P}(M_0) \,, $$
showing that  the distance between the projection of the surface $\eta(\partial \omega_- \cap \{x_d=0\}, t)$ onto the plane $x_d=0$  and
the set  $\partial \omega_- \cap \{x_d=0\}$ is $O( \epsilon )$.   Since by  Definition \ref{def-dino-e}, the set $\partial \omega_- \cap \{x_d=0\}$ contains
a $d$$-$$1$-dimensional ball of radius $ \sqrt{ \epsilon }$ centered at the origin, we see that by choosing $ \epsilon $ sufficiently small
 the vertical line passing through $\eta(X_+^ \epsilon , t )$ must intersect the surface $\eta(\partial \omega_- \cap \{x_d=0\}, t)$ 
 for any $t\in [0,10\epsilon]$.  Now, since at $t=0$,  $X_+^\epsilon$ is directly (vertically) above $\partial \omega_- \cap \{x_d=0\}$,  and at  
 $t= 10 \epsilon$, from (\ref{cs201}),  
$\eta(X_+^ \epsilon , 10 \epsilon )$ is (vertically)  below $\eta(\partial \omega_- \cap \{x_d=0\},10\epsilon)$, then by continuity there  necessarily exists 
a time $0< T^* < 10 \epsilon $ at which $\eta(X_+^ \epsilon , T^* ) = \eta(Z , T^* )$ for some $Z \in
\textcolor{black}{\partial \omega_-\cap\{x_d=0\}}$. This concludes the proof of the main
theorem.

%so there exists a time $0< T^* < 10 \epsilon $ at which self-intersection of $\Gamma^ \epsilon (T^*)$ must occur. 

\section{The case of a general self-intersection splash geometry}\label{section9}
We now show how the analysis presented in the previous sections for the case of the ``dinosaur wave'' initial domain can be used to establish the 
existence of a splash singularity in a finite time $T^*$ for any domain whose boundary is arbitrarily close (in the $H^3$-norm) to any 
given self-intersecting surface of class $H^3$.   This generalization requires the  geometric constructions that we introduced in our previous work \cite{CoSh2014a}, coupled with a very minor adaptation of  the analysis of the previous sections.

We begin with the  definition of the splash domain that we gave in \cite{CoSh2014a}.

\subsection{The definition of the splash domain}\label{subsec:splashdomain}

\begin{enumerate}
\item
We suppose that $x_0  \in \Gamma:= \partial \Omega_s$ is the unique boundary self-intersection point,
 i.e., $\Omega_s$ is locally on each side of the tangent plane to $\partial\Omega_s=\Gamma_s$ at $x_0$.
 For all other boundary points, the domain is locally on one side of its boundary.   Without loss of
 generality, we suppose that
the tangent plane at $x_0$ is the horizontal plane $x_3-(x_0)_3=0$. 

\item We let $U_0$ denote an open neighborhood of $x_0$ in $ \mathbb{R}  ^3$, and then choose an additional $L$ open
sets $\{U_l\}_{l=1}^L$ such that the collection  $\{U_l\}_{l=0}^K$ is an open cover of $\Gamma_s$, and $\{U_l\}_{l=0}^L$ is an open cover of $\Omega_s$ and such that there exists a
sufficiently small open subset $ \omega \subset U_0$ containing $x_0$ with the property that 
$$\overline\omega \cap \overline{U_l} = \emptyset \ \text{ for all } \ l=1,...,L \,.$$
We set
\begin{align*} 
U_0^+ = U_0 \cap \Omega_s \cap \{ x_3 > (x_0)_3 \} \ \text{ and } U_0^- = U_0 \cap \Omega_s \cap \{ x_3 < (x_0)_3 \} \,.
\end{align*} 
Additionally, we assume that $\overline{U_0}\cap\overline{\Omega_s}\cap\{x_3=(x_0)_3\}=\{x_0\}$, which implies in particular that $U_0^+$ and $U_0^-$ are connected.   See Figure 9.1.

\begin{figure}[htbp]
\begin{center}
\includegraphics[scale = 0.4]{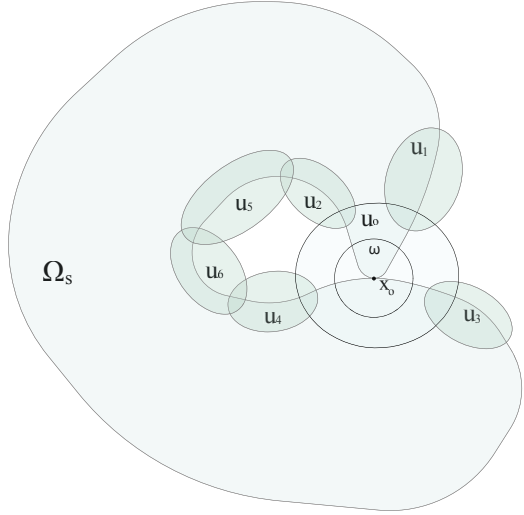}
\caption{Splash domain $\Omega_s$, and the collection of open set $\{U_0,U_1,U_2,...,U_K\}$ covering $\Gamma$.}
\end{center}
\label{fig3}
\end{figure}

\item For each $l\in \{1,...,K\}$, there exists an  $H^{3}$-class diffeomorphism $\theta_l$  satisfying
\begin{gather}
\theta_l : B:=B(0,1) \rightarrow U_l \nonumber \\
U_l \cap \Omega_s = \theta_l ( B^+ )
\ \text{ and } \ \overline{U_l} \cap \Gamma_s = \theta_l ( B^0 ) \,, 
\nonumber
\end{gather}
where 
\begin{align*} 
B^+ &=\{(x_1,x_2,x_3)\in B:  x_3>0\} \,, \\
B^0 &=\{(x_1,x_2,x_3)\in \overline B: x_3=0\}\,.
\end{align*} 

\item For $L > K$, let $\{U_l\}_{l=K+1}^{L}$ denote a family of open sets 
contained in $\Omega_s$ such that 
$\{U_l\}_{l=0}^{L}$ is an open cover of $\Omega_s$, and for $l\in \{K+1,...,L\}$, $\theta_l : B \to U_l$ is an
$H^{3}$ diffeormorphism.

\item To the open set $U_0$ we associate two $H^{3}$-class diffeomorphisms $\theta_+$ and $\theta_-$ of $B$ onto $U_0$ with the following properties:

\begin{alignat*}{2}
\theta_+(B^+) &= U_0^+ \,,          \qquad \qquad           && \theta_-(B^+)= U_0^-   \,,   \\
\theta_+(B^0) & = \overline{U_0^+}\cap \Gamma_s\,,       &&  \theta_-(B^0) = \overline{U_0^-}\cap \Gamma_s\,,
\end{alignat*}
such that
\begin{equation}\nonumber
\{x_0\}=\theta_+(B^0)\cap\theta_-(B^0)\,,
\end{equation} 
and
\begin{equation}\nonumber
\theta_+(0)=\theta_-(0)=x_0\,.
\end{equation}  
We further assume that 
$$ \overline{\theta_\pm(B^+\cap B(0,1/2))} \cap \overline{\theta_l(B^+)} = \emptyset \text{ for } l=1,...,K \,,$$
and
$$ \overline{\theta_\pm(B^+\cap B(0,1/2))} \cap \overline{\theta_l(B)} = \emptyset \text{ for } l=K+1,...,L \,.$$
\end{enumerate}

 \begin{definition}[Splash domain $\Omega_s$]\label{def:splashdomain} 
 We say that $\Omega_s$ is a splash domain, if it is defined by a collection of
 open covers $\{U_l\}_{l=0}^L$ and associated maps $\{\theta_\pm, \theta_1, \theta_2,...,\theta_L\}$ satisfying the
 properties (1)--(5) above.   Because each of the maps is an $H^{3}$ diffeomorphism, we say
that the splash domain $\Omega_s$ defines a self-intersecting {\it generalized} $\bf H^{3}$-domain.
 \end{definition} 

\subsection{An approximating sequence of non self-intersecting domains converging to the splash domain}
 Following \cite{CoSh2014a}, we can then define standard (non self-intersecting) domains $\Omega^\epsilon$ (for $\epsilon>0$ small enough) by just 
 modifying $\theta_\pm$, and leaving the other charts unchanged.  As shown in Figure \ref{fig4}, our non self-intersecting domain $\Omega^\epsilon$ 
 will be defined by associated maps $\{\theta\epsilon_\pm, \theta_1, \theta_2,...,\theta_L\}$ such that 
 \begin{equation}
 \label{thetaeps}
 \|\theta^\epsilon_\pm-\theta_\pm\|_{H^3(B^+)}\le C\epsilon\,,
 \end{equation} 
 and such that  
 \begin{equation}
 \label{thetaepsbis}
 0<d(\theta^\epsilon_+(B^+),\theta^\epsilon_-(B^+))\le \epsilon\,.
 \end{equation}
 
 \begin{figure}[htbp]
\begin{center}
\includegraphics[scale = 0.55]{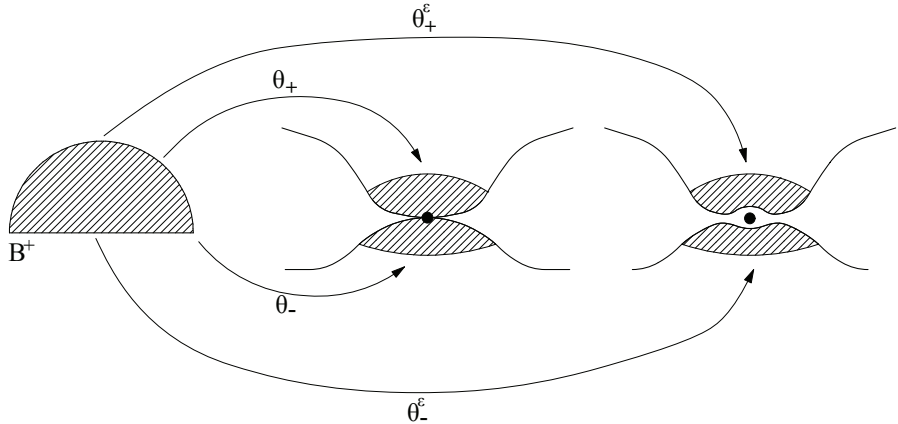}
\caption{The black dot denotes the point $x_0$ where the boundary self-intersects (middle).  For $ \epsilon >0$, the approximate domain $\Omega^ \epsilon $ does not intersect itself (right).}
\end{center}
\label{fig4}
\end{figure}

 In summary, we have approximated the self-intersecting splash domain $\Omega_s$ with a sequence of $H^{3}$-class 
 domains $\Omega^ \epsilon $ converging toward $\Omega$,  such that for each $ \epsilon >0$, $\p \Omega ^ \epsilon $ does not self-intersect.
  As such, 
 each one of these domains $\Omega^ \epsilon $, $ \epsilon >0$,  will thus be amenable to our local-in-time well-posedness theory 
 for free-boundary incompressible Navier-Stokes equations.
 
\section{Existence of a splash in finite time in a domain arbitrarily close to a given splash domain} \label{section10}
 
We next define an initial velocity field of the same type as in Section \ref{sec::u0}. Due to (\ref{thetaeps}), the estimates of Section \ref{section7} remain 
unchanged. Similarly, the main proof of Section \ref{section8} works in a similar manner due to (\ref{thetaepsbis}), leading to the necessity of
 self-intersection at a time $T^\epsilon\in (0,10 \epsilon)$.    Note that since the tangent plane at the intended splash singularity $x_0$  is the horizontal plane 
 $\{ x_3=0\}$,  $\partial [\theta_-( B^+)]$ is very close to $\{x_3=0\}$ in a small ball $B(x_0, \sqrt{ \epsilon })$ for $\epsilon $ taken sufficiently small;
 thus, we are 
 using the fact that the almost flat portion of $\theta_-( B^+)$ is very close to $\{ x_3=0\}$ and contains a region of diameter at least $ \sqrt{ \epsilon }$.

 Furthermore, 
\begin{align}
\|\eta^\epsilon(\theta^\epsilon_\pm, T^\epsilon)-\theta_\pm\|_3& \le  \|\eta^\epsilon(\theta^\epsilon_\pm, T^\epsilon)-\theta^\epsilon_\pm\|_3+\|\theta^\epsilon_\pm-\theta_\pm\|_3\nonumber\\
& \le \|\int_0^{T^\epsilon} v^\epsilon(\theta^\epsilon_\pm,t)\ dt\|_3+C\epsilon\,,
\label{9.1}
\end{align}
where we used  the estimate (\ref{thetaeps}) in the above inequality (\ref{9.1}); hence,   from our estimates in Section \ref{section7},  
\begin{equation}
\|\eta^\epsilon(\theta^\epsilon_\pm,T^\epsilon)-\theta_\pm\|_3  
 \le C\mathcal{P}(M_0) \sqrt{T^\epsilon}+C\epsilon\le C\mathcal{P}(M_0) \sqrt{\epsilon}\,.
\label{9.2}
\end{equation}
 
This, therefore, shows that the splash-free surface $\eta^\eps(\Omega^\epsilon,T^\epsilon)$ is at a distance less than $C\mathcal{P}(M_0) \sqrt{\epsilon}$ from $\Omega_s$ in $H^3$. We have then established the following:

\begin{theorem}\label{thm_general}
For {\it any}  given splash domain $\Omega_s$ of class $H^3$, there exists a splash domain $\tilde \Omega_s$ arbitrarily close in $H^3$
 to $\Omega_s$, and
smooth initial data consisting of a  non self-intersecting domain $\Omega^\epsilon$ of class $H^{3}$ and a divergence-free velocity field 
$u_0^ \epsilon  \in H^3( \Omega^\epsilon )$ satisfying $ [ \operatorname{Def} u_0^ \epsilon \cdot N_\epsilon ] \times N_ \epsilon =0$ on $\partial\Omega^\epsilon$, such that the
flow map $\eta(x, t)$ solving the Navier-Stokes equations (\ref{NSlag}) satisfies $\eta( \p \Omega^\epsilon , T^*) = \tilde \Omega_s$.  That is, in
finite time $T^*>0$, a splash singularity occurs which is very close to a prescribed self-intersecting geometry.
\end{theorem}

\section*{Acknowledgments}
DC was supported by the Centre for Analysis and Nonlinear PDEs funded by the UK EPSRC grant EP/E03635X and the Scottish Funding Council.  SS was supported by the National Science Foundation under grant DMS-1301380
and by the Royal Society Wolfson Merit Award.


\begin{thebibliography}{50}


\bibitem{Abels2005} H.~Abels, \emph{The initial-value problem for the {N}avier-{S}tokes equations with a
  free surface in {$L^q$}-{S}obolev spaces}, 
 Adv. Differential Equations, {\bf 10}, (2005), 45--64.


\bibitem{AmGi1991} C. Amrouche and V. Girault, \emph{The existence and regularity of the solution of Stokes problem in arbitrary dimension},
Proc. Japan Acad., {\bf 67}, Ser. A,  (1991), 171--175.

\bibitem{AmSe2011} C. Amrouche and N.E.H. Seloula, \emph{On the Stokes equations with the Navier-type boundary conditions},
 Differ. Equ. Appl., {\bf 3}  (2011), 581--607

\bibitem{Bae2011}  H. Bae, \emph{Solvability of the free boundary value problem of the Navier-Stokes equations},  Discrete
Contin. Dyn. Syst.,  {\bf 29}, (2011), 769--801.

\bibitem{Beale1981} J. Beale.  \emph{The initial value problem for the Navier-Stokes equations with a free surface}, 
 {Comm. Pure Appl. Math.} \textbf{34} (1981), no. 3, 359--392. 

\bibitem{Beale1983} J.T.~Beale, \emph{Large-time regularity of viscous 
surface waves,}  Arch. Rational Mech. Anal., {\bf 84}  (1983/84), 307--352.



\bibitem{CaCoFeGaGo2012} A. Castro, D. C\'{o}rdoba, C. Fefferman, F. Gancedo, and M. G\'{o}mez-Serrano, \emph{Finite time singularities
for water waves with surface tension},  Journal of Mathematical Physics,  {\bf 53}, (2012), 115622--115622.


\bibitem{CaCoFeGaGo2013} A. Castro, D. C\'{o}rdoba, C. Fefferman, F. Gancedo, and M. G\'{o}mez-Serrano,
\emph{Finite time singularities for the free boundary incompressible Euler equations}, Ann. of Math., {\bf  178},  (2013),  1061--1134.

\bibitem{CaCoFeGaGo2015} A. Castro, D. C\'{o}rdoba, C. Fefferman, F. Gancedo, and M. G\'{o}mez-Serrano,
\emph{Splash singularities for the free boundary Navier-Stokes equations},  (2015), \emph{arXiv:1504.02775}.

\bibitem{CoEnGr} D. C\'{o}rdoba, A. Enciso, and N. Grubic, \emph{Splash and almost-splash stationary solutions to the Euler
equations}, (2014), \emph{ Arxiv preprint arXiv:1412.7382}.


\bibitem{ChSh2010}  C.H.A.~Cheng and S.~Shkoller, \emph{The interaction of the 3D Navier-Stokes equations with a moving nonlinear
 Koiter elastic shell},  SIAM J. Math. Anal., {\bf 42},  (2010), 1094--1155.


\bibitem{CoSh2002} D.~Coutand, S.~Shkoller, \emph{Unique solvability of the free-boundary Navier-Stokes equations with surface tension}, (2002),
\emph{arXiv:math/0212116}.

\bibitem{CoSh2005} D.~Coutand, S.~Shkoller, \emph{  On the motion of an 
elastic solid inside of an incompressible viscous fluid,}  Arch. Rational 
Mech. Anal., {\bf 176}, (2005), 25--102.

\bibitem{CoSh2006} D.~Coutand, S.~Shkoller, \emph{  The interaction between
quasilinear elastodynamics and the Navier-Stokes equations,}
Arch. Rational Mech. Anal., {\bf 179},  (2006),  303--352.

\bibitem{CoSh2007}
D.~Coutand and S.~Shkoller,
\emph{ Well-posedness of the free-surface incompressible Euler
equations with or without surface tension},
 J. Amer. Math. Soc., {\bf 20},  (2007), 829--930.
 


 \bibitem{CoSh2014a} D.~Coutand and S.~Shkoller, \emph{On the Finite-Time Splash and Splat Singularities for the 3-D Free-Surface Euler Equations},
  Comm. Math. Phys., {\bf 325}, (2014),  143--183.
  
 \bibitem{CoSh2014b} D.~Coutand and S.~Shkoller, \emph{On the impossibility of finite-time splash singularities for vortex sheets}, (2014), 
 \emph{arXiv:1407.1479}.



\bibitem{ElLe2014}
T.~Elgindi and D.~Lee, \emph{Uniform regularity for free-boundary {N}avier-{S}tokes equations with
  surface tension}, (2014),  \emph{arXiv:1403.0980}.

\bibitem{FeIoLi2013} C.~Fefferman, A.D.~Ionescu, and V.~Lie, \emph{On the absence of ``splash'' singularities in the case of two-fluid interfaces}, Preprint,
(2013), \emph{arXiv:1312.2917}.

\bibitem{Hataya2009} Y. Hataya, \emph{Decaying solution of a Navier-Stokes flow without surface tension}, 
J. Math. Kyoto Univ.,  \textbf{49} (2009), 691--717.  


\bibitem{GuTi2013b}
Y.~Guo and I.~Tice, \emph{Almost exponential decay of periodic viscous surface waves without
  surface tension}, Arch. Ration. Mech. Anal.,  {\bf 207}, (2013), 459--531.

\bibitem{GuTi2013c}
Y.~Guo and I.~Tice, \emph{Decay of viscous surface waves without surface tension in
  horizontally infinite domains}, 
 Anal. PDE {\bf  6}, (2013), 1429--1533.

\bibitem{GuTi2013a}
Y.~Guo and I.~Tice, \emph{Local well-posedness of the viscous surface wave problem without
  surface tension},  Anal. PDE, {\bf 6}, (2013), 287--369.

\bibitem{MaRo2012} N.~Masmoudi and F.~Rousset,
\emph{Uniform regularity and vanishing viscosity limit for the free surface
  {N}avier-{S}tokes equations},  (2012),  \emph{arXiv:1202.0657}.

\bibitem{NiTeYo2004} T. Nishida, Y. Teramoto, H. Yoshihara. Global in time behavior of viscous surface waves: horizontally periodic motion. \emph{J. Math. Kyoto Univ.} \textbf{44} (2004), no. 2, 271--323. 

\bibitem{PaSo2000} M. Padula and V.A. Solonnikov, 
\emph{On Rayleigh-Taylor stability},  Navier-Stokes equations and related nonlinear problems (Ferrara, 1999), 
Ann. Univ. Ferrara Sez. VII (N.S.), {\bf 46},  (2000), 307--336. 

\bibitem{Sol1977}
V. A. Solonnikov, 
\emph{Solvability of the problem of the motion of a viscous incompressible fluid that is bounded by a free surface},  
Izv. Akad. Nauk SSSR Ser. Mat., {\bf 41}  (1977),  1388--1424.

\bibitem{Sol1991} V.A.~Solonnikov,
\emph{ On an initial boundary value problem for the Stokes systems arising
in the study of a problem with a free boundary,} Proc. Steklov Inst. Math.,  
{\bf 3} (1991), 191--239.

\bibitem{Sol1992} V.A.~Solonnikov,
\emph{  Solvability of the problem of evolution of a viscous incompressible
fluid bounded by a free surface on a finite time interval,}
St. Petersburg Math. J., {\bf 3} (1992), 189--220.


\bibitem{SoSc1973}
V.A. Solonnikov and V.E. Scadilov, 
\emph{On a boundary value problem for a stationary
system of Navier-Stokes equations} , Proc. Steklov Inst. Math.,  {\bf 125}, (1973), 186--199.

\bibitem{TaTa1995} A. Tani, N. Tanaka,  \emph{ Large-time existence of surface waves in incompressible viscous fluids with or without surface tension},  Arch. Rational Mech. Anal.,  \textbf{130}  (1995),  no. 4, 303--314.

\bibitem{WaXi2015} Y. Wang and  Z. Xin, \emph{ Vanishing viscosity and surface tension limits of incompressible viscous surface waves}, (2015), 
  \emph{arXiv:1504.00152}.


\end{thebibliography}
\end{document}